\documentclass[11pt]{article}
\usepackage[utf8]{inputenc}
\usepackage{lmodern}
\usepackage{subfiles}
\usepackage{enumitem}
\setenumerate{topsep=6pt,ref={\normalfont(\roman*)},label={\normalfont(\roman*)}, itemsep=0pt} % or \upshape

\usepackage{amsfonts}
\usepackage{amsthm}
\usepackage{amsmath}
\usepackage{amssymb}
\usepackage{amscd}
\usepackage{mathrsfs}
\usepackage{mathtools}
\usepackage{bbm}
\usepackage{esint}

\usepackage[margin=2.7cm]{geometry}
\usepackage{setspace}
\usepackage{indentfirst}
\usepackage{graphicx}
\usepackage{graphics}
\usepackage{lscape}
\usepackage{pgf,tikz}
\usepackage{tikz-cd}
\usepackage{color}
\usepackage{pict2e}
\usepackage{epic}
\usepackage{epstopdf}
\usepackage{titlesec, titlefoot}
\titleformat{\section}[block]{\Large\bfseries\filcenter}{\thesection}{1em}{}
\titleformat{\part}[block]{\LARGE\bfseries\filcenter}{Part \thepart.}{0.5em}{}
\usepackage{commath}
\usepackage{float}
\usepackage{caption}
\usepackage{etoolbox}
\usepackage[affil-it]{authblk}
\usepackage{combelow}

\usepackage[hidelinks, bookmarksdepth=3,citecolor=blue,colorlinks]{hyperref}
\hypersetup{bookmarksopen=true} 

\graphicspath{{./Pictures/}}
\allowdisplaybreaks

\expandafter\def\expandafter\normalsize\expandafter{%
\normalsize
\setlength\abovedisplayskip{6pt}
\setlength\belowdisplayskip{6pt}
\setlength\abovedisplayshortskip{6pt}
\setlength\belowdisplayshortskip{6pt}
}

\theoremstyle{plain}

\renewcommand*\thesection{\arabic{section}}
\numberwithin{equation}{section} 

\newtheorem{theorem}{Theorem}[section]
\newtheorem{lemma}[theorem]{Lemma}
\newtheorem*{lemma*}{Lemma}
\newtheorem{proposition}[theorem]{Proposition}
\newtheorem{corollary}[theorem]{Corollary}

\theoremstyle{definition}
\newtheorem{definition}[theorem]{Definition}
\newtheorem{remark}[theorem]{Remark}

\expandafter\let\expandafter\oldproof\csname\string\proof\endcsname
\let\oldendproof\endproof
\renewenvironment{proof}[1][\proofname]{%
\oldproof[\upshape \bfseries #1]%
}{\oldendproof}
\makeatletter
\def\@makechapterhead#1{%
\vspace*{50\p@}%
{\parindent \z@ \raggedright \normalfont
\interlinepenalty\@M
\Huge\bfseries  \thechapter.\quad #1\par\nobreak
\vskip 40\p@
}}
\makeatother

\renewcommand{\Re}{\operatorname{Re}}
\renewcommand{\Im}{\operatorname{Im}}

\newcommand{\eps}{\varepsilon}

\DeclareMathOperator{\ddiv}{div}

\DeclareMathOperator{\Id}{Id}
\DeclareMathOperator{\cof}{cof}
\DeclareMathOperator{\tr}{tr}
\DeclareMathOperator{\diag}{diag}
\DeclareMathOperator{\Diag}{Diag}
\DeclareMathOperator{\Sym}{Sym}

\DeclareMathOperator{\SO}{SO}
\DeclareMathOperator{\OO}{O}
\DeclareMathOperator{\rank}{rank}

\DeclareMathOperator{\Lip}{Lip}

\def \a{\alpha}

\def \R {\mathbb{R}}

\def \C{\mathbb{C}}

\def \N{\mathbb{N}}
\def \D{\textup{D}}

\def \e{\varepsilon}
\def \d{\,\textup{d}}

\def \p{\partial}

\def \mb{\mathbb}

\def \w{\rightharpoonup}

\def \tp{\textup}
\def \Id{\textup{Id}}

\def \loc{\textup{loc}}
\def \CO{\mathcal{H}}
\def \ACO{\overline{\mathcal{H}}}

\begin{document}

	\title{\textbf{Stationary points of conformally invariant\\ polyconvex energies}}
	
%\author{}	
	
	\author[1]{{\Large Andr\'e Guerra}}
	\author[2]{{\Large Riccardo Tione}}
		
	\affil[1]{\small Department of Pure Mathematics and Mathematical Statistics,  University of Cambridge,\protect\\  Wilberforce Rd, Cambridge CB3 0WB, UK
	\protect\\
	{\tt{adblg2@cam.ac.uk}} \vspace{1em} \ }
		
	\affil[2]{\small Dipartimento di Matematica, Universit\`a degli Studi di Torino,
Via Carlo Alberto 10, 10123 Torino, Italy 
	\protect\\
	{\tt{riccardo.tione@unito.it}}  }
	
	\date{}
	
	\maketitle

	\unmarkedfntext{
	\hspace{-0.75cm}
	%\emph{2020 Mathematics Subject Classification:} 35F30 (49J45).\\ 
	%\emph{Keywords.} Jacobian, Underdetermined PDE, Nonlinear PDE, Measurable data, Radial stretchings.\\
	\emph{Acknowledgments.} AG acknowledges the support of the Royal Society through a Newton International Fellowship. RT acknowledges the hospitality of the University of Cambridge, where part of this research was conducted.\ RT was funded by the Italian Ministry for University and Research (MUR), through FIS3 Starting Grant “GEMS” CUP: D53C25002540001 (Finanziata con il contributo del Ministero dell’Università e della ricerca ai sensi del D.D. n. 1802 del 21-11-2024 - BANDO FIS 3, Grant number FIS-2024-02219). We thank J. Hirsch for suggesting Lemma \ref{lemma:sigmasys}.\ 
	}
	
%	\vspace{1cm}	
	
	\begin{abstract}
We consider polyconvex integrands that are conformally invariant and frame indifferent. In two dimensions, we prove that the corresponding stationary points are smooth outside a discrete set; this result is new even for minimizers. We further show that every orientation-preserving stationary point is $C^1$. Since such solutions are closely related to Teichmüller-type variational problems, our result also confirms, in the case of integrands with linear growth in the distortion, a conjecture of Astala, Iwaniec, Martin, and Onninen from 2005.
	\end{abstract}

%	\tableofcontents	
	
	\section{Introduction}
	In this paper we are concerned with variational integrals
	$$\mb E[u] \equiv \int_\Omega F(\D u) \d x, \qquad u\colon \Omega\to \R^2,$$	
	where $F\in C^1(\R^{2\times 2})$ and $\Omega\subset \R^2$ is a planar domain.  
%	 assumption that the target is two dimensional is ma although under appropriate modifications the ensuing discussion also holds in higher dimensions.	
	Provided that $F$ has quadratic growth, any $W^{1,2}$-minimizer $u$ of $\mb E$ is necessarily a critical point of $\mb E$ with respect to target and domain variations \cite{Giaquinta2004}, i.e.\ it is a weak solution of the systems
%	whose critical points satisfy the Euler--Lagrange system
%	\begin{equation}
%	\label{eq:EL-general}
%	\ddiv \p_A F(u,\D u)=\p_{u} F(u,\D u).
%	\end{equation}
%	We assume that $F\in C^1( \R^m\times \R^{m\times 2})$ is a \textit{coercive integrand with quadratic growth}, i.e.\
%	\begin{align*}
%	-c_0+c_1 |A|^2  \leq  F(u,A) \leq c_0 +c_2 |A|^2, \qquad 
%	|\p_u F(u,A)| + |\p_A F(u,A)|^2 \leq c_2 (1+|A|^2),
%	\end{align*}
%	for some constants $0<c_1\leq c_2, 0\leq c_0$ and for all $(u,A)\in \R^m\times \R^{m\times 2}$.
	\begin{align}
	\label{eq:EL}
	\ddiv(F'(\D u))=0,\\
	\label{eq:inner}
	\ddiv(T(\D u))=0.
	\end{align}
	Here, and throughout the paper, for $A\in \R^{2\times 2}$ the \textit{stress-energy tensor} $T(A)$ is defined as
	\begin{equation}
	\label{eq:defT}
	T(A) \equiv F(A) \Id- A^T F'(A).
	\end{equation}
	We refer to weak solutions of \eqref{eq:EL}--\eqref{eq:inner} as \textit{stationary points} of $\mb E$.

	This paper deals with the regularity theory of stationary points when $F$ is strongly \textit{polyconvex}, i.e.\ when $F(A)$ is a strongly convex function of $(A,\det A)$, and we begin by briefly reviewing it.
%	 We begin by briefly reviewing the state of the

%	\textcolor{blue}{New structure: start with a good introduction to the theory of pc functionals, in the first subsection.\ I would not focus too much here on the $C^{1,1}$ vs $C^2$ regularity, this cane be said after the statement of the main theorem. Then, in "setting and main results" one can mention Hild and how he gives us a an idea for a good, large family of functionals. In passing, one cites harmonicas (for a couple of sentences). Then the main result follows and after we can put C11 vs C2 functionals. Next second result and we are done. In both cases, a sentence with "this will be done using a method that goes back to Sverak/ Iwaniec\&al" can surely be added, but maybe I wouldn't make the intro too too heavy on the math details. Also add the connection to the problem in the AIM book.}

\subsection{Brief overview of the regularity theory for polyconvex integrands}
\label{sec:regtheory}

It is known since Morrey's work that, if $F$ is regular and strongly convex, then minimizers of $\mb E$ are smooth. In constrast, comparatively little is known when $F$ is strongly polyconvex. Polyconvexity is a natural condition for the existence of minimizers \cite{Dacorogna2007} and it is verified by almost  all the relevant examples arising in Nonlinear Elasticity \cite{Ball1977}, Geometric Function Theory \cite{Iwaniec2001}, and Geometric Measure Theory \cite{DeLellis2019}, underscoring its importance in vectorial variational problems.

Unlike the convex case, where \eqref{eq:EL} is equivalent to minimality of $u$, this equivalence fails for polyconvex $F$, even if $u$ is smooth \cite{Spadaro2009}. This discrepancy leads to important distinctions between three classes of solutions: (i) \textit{minimizers} of $\mb E$; (ii) weak solutions of \eqref{eq:EL}--\eqref{eq:inner}, i.e.\ \textit{stationary points};  and (iii) weak solutions of \eqref{eq:EL} alone, which we refer to as \textit{weak extremals}.
We briefly summarize what is known for each class.
\begin{enumerate}
\item \textit{Minimizers}: if $F\in C^2$ then minimizers are smooth outside a closed  set $\Sigma$ of measure zero \cite{Acerbi1987,Evans1986}. Moreover, if a minimizer is Lipschitz, then $\Sigma$ is uniformly porous  \cite{Kristensen2007} and in particular has Hausdorff dimension strictly less than 2. We emphasize that, in the two-dimensional setting considered here, there are no examples of minimizers with $\Sigma \neq \emptyset$.
\item \textit{Stationary points}: there is no definitive general result available. A key difficulty is that the class of stationary points is not compact, at least for high-dimensional targets \cite{Muller1999a,Sverak2025}.\ On the other hand, it is possible to exclude the most common convex integration counterexamples \cite{DeLellis2019,Hirsch2021a} that we will mention in the next point of this list. 
In the specific case where $F$ is the area integrand  \cite{Giaquinta2012} or a small perturbation of it \cite{Rosa2020,Tione2021} (partial) regularity theorems are available. Beyond this regime, no further positive results are known.
\item \textit{Weak extremals}: there can be no general regularity theory. Indeed, the convex integration constructions of \cite{Muller2003,Szekelyhidi2004} produce strongly polyconvex integrands with Lipschitz weak extremals which are nowhere $C^1$. The gap between Lipschitz and $C^1$ regularity is therefore fundamental: by Schauder theory, any $C^1$ weak extremal is as regular as the integrand, e.g.\ smooth if $F$ is smooth \cite{Giaquinta2012}. The area integrand is again exceptional in this respect, as its weak extremals are smooth  \cite{Hirsch2023}. No other positive results are known, although  \cite{GuerraTione2024,Tione2022} provide a family of examples where a form of degenerate regularity is available.
\end{enumerate}

\subsection{Partial regularity for stationary points of conformally invariant integrands}

The purpose of this paper is to present a large new class of natural examples whose stationary points are $C^1$. We follow Hildebrandt's conjecture \cite{Hildebrandt1982}  that such regularity should hold for energies which are \textit{conformally invariant}, a ubiquitous condition in geometric problems: precisely, whenever $\Omega'\subset \R^2$ is open and $\varphi\colon \Omega'\to \Omega$ is a conformal diffeomorphism, then
	\begin{equation}\label{eq:ci}
		\mb E[u]= \mb E[u\circ \varphi] .
	\end{equation}
This invariance yields quite a rigid structure, since it forces the integrand to be 2-homogeneous, see \cite{GuerraKristensen2021} for automatic polyconvexity results for such integrands. In particular,  if $F\in C^2$ then $F$ is quadratic and therefore \eqref{eq:EL} is linear. A more interesting question arises if 
$F$ depends not just on $\D u$ but also on $u$ itself, as is the case for harmonic maps or $H$-systems.  Smooth conformally invariant integrands of this type were classified in \cite{Gruter1984} and a deep regularity theory for weak extremals was obtained in \cite{Riviere2007}, see also \cite{Riviere2008,Schikorra2010,Sharp2013} for further developments and \cite{Mazowiecka2018b,Schikorra2015} for nonlocal conformally-invariant analogues.

In contrast, here we do not assume that $F$ is quadratic (and hence $F$ is not $C^2$). However, we impose the condition, standard in Continuum Mechanics \cite{Marsden1994}, that $\mb E$ is \textit{frame independent}, meaning that the energy of the deformed configuration $u(\Omega)$ is independent of the observer's frame of reference. Equivalently,
\begin{equation}\label{eq:fi}
\text{$\mb E[Q\circ u]=\mb E[u]$ for any isometry $Q$.}
\end{equation}
As a simple example of the type of energies that we will consider, take the integrands
\begin{equation}\label{eq:ex1}
F_\e (A) = \sqrt{\e |A|^4 + \det(A)^2}, \qquad \e>0.
\end{equation}
The	corresponding energies satisfy \eqref{eq:ci}--\eqref{eq:fi} and, for $\e>0$ small, $F_\e$ is approximately $\lvert \det(\cdot) \rvert$, which is manifestly non-convex.\ 

More generally, as we show in Section \ref{sec:prelims}, \eqref{eq:ci}--\eqref{eq:fi} are equivalent to the representation
	\begin{equation}
	\label{eq:repF}
	F(A)= f(|A|^2, \det A) , \qquad \text{ where } f(t n, t d) = f(t n, -t d) = t f(n,d),
	\end{equation}
	for all $A\in \R^{2\times 2}$ and $t\geq 0$. Here and throughout, $|A|$ denotes the Euclidean norm of $A$.
The requirement of strong polyconvexity then leads us to assume that
\begin{equation}
\label{eq:ell}
f \text{ is smooth away from 0, convex and }
\begin{cases}
f_1 \geq \nu>0\\
%& \text{ if } n\geq 2 |d| \text{ and } n >0,\\ 
f_2\geq 0
%& \text{ if } n\geq 2 d\geq 0 \text{ and } n >0.
\end{cases}
\hspace{-.2cm} \text{ in } \{n\geq 2 d \geq 0, n>0\}.
\end{equation}
In particular, through \eqref{eq:repF}, example \eqref{eq:ex1} corresponds to the choice $f(n,d)=\sqrt{\e n^2 + d^2}$. More generally, we can take $f$ to be any smooth, uniformly convex norm in $\R^2$ with $f(n,d)=f(|n|,|d|)$.
		
		The first result of this paper is a precise partial regularity theorem for the above class:
	
	\begin{theorem}\label{thm:partialreg}
	Let $u\in W^{1,2}(\Omega,\R^2)$ be a stationary point of $\mb E$ and assume \eqref{eq:repF}--\eqref{eq:ell} hold.\ Then either $u$ is harmonic or $u\in C^\infty(\Omega\setminus \Sigma), \text{where } \Sigma \subset \Omega \text{ is discrete}.$	Moreover, there is $s>\frac 3 2$ such that $u\in W^{s,2}_\loc(\Omega,\R^2)$, and whenever $s<\tfrac 3 2$ we further have the uniform estimate
\begin{equation}
\label{eq:Ws2}
\| u\|_{W^{s,2}(\Omega')} \leq C(s,\Omega',\Omega)\, \| u\|_{W^{1,2}(\Omega)}.
\end{equation}
	\end{theorem}
	
We now make a series of remarks concerning this theorem.

\begin{remark}[Optimality]
Since $F\in C^{1,1}$ and  $F\not \in C^2$, Theorem \ref{thm:partialreg} is likely optimal, i.e.\ we do not expect stationary points to be smooth everywhere.
\end{remark}
	
\begin{remark}[Minimizers]
Theorem \ref{thm:partialreg} is new already for minimizers. As discussed above, for $C^2$ integrands the best results available yield only $\mathscr L^2(\Sigma)=0$. Moreover, these results are perturbative, as they rely crucially on linearizing \eqref{eq:EL}, and therefore for $C^{1,1}$-integrands, such as the ones we consider here, it is not even known if $\mathscr L^2(\Sigma)=0$; see also \cite{Duzaar2004}. Note that this state of affairs is in striking contrast with the scalar case, where the non-perturbative De Giorgi--Nash--Moser theory applies even if $F$ is merely $C^{1,1}$, see \cite{Figalli2017b} for further discussion.
\end{remark}	

\begin{remark}[$\e$-regularity and compactness]\label{rem:e-reg}
Theorem \ref{thm:partialreg} is not an $\e$-regularity theorem, unlike e.g.\ the regularity theorems for weakly harmonic maps in the plane \cite{Giaquinta2012}. Further to this point, note that \eqref{eq:Ws2} implies that a sequence of stationary points with bounded energy is compact in $W^{1,2}_\loc$. On the contrary, sequences of smooth, stationary harmonic maps with bounded energy may be non-compact in $W^{1,2}_\loc$ and in general develop bubbling \cite{Parker1996}.
\end{remark}

	\subsection{Application to Teichm\"uller problems}
		
	The conformally invariant integrands considered in Theorem \ref{thm:partialreg} arise naturally in Geometric Function Theory, as we now explain. A map $u\in W^{1,1}_\loc(\Omega,\R^2)$ is said to have \textit{finite distortion} if there is a measurable function $\mb K\colon \Omega\to [1,\infty)$ such that
$$\tfrac 1 2 |\D u(x)|^2 \leq  \mb K(x) \det \D u(x)\quad \text{ for a.e.\ } x\in \Omega.$$
We denote by $\mb K_u$ the least such function $\mb K$. 
When $\mb K_u\in L^\infty$ we say that $u$ is \textit{quasiregular}, and if moreover $u$ is a homeomorphism we say that $u$ is \textit{quasiconformal}.

It is often desirable to work with mappings of finite distortion. For instance, a planar orientation-preserving homeomorphism $u\in W^{1,1}_\loc$ has an inverse $f\in W^{1,1}_\loc$ if and only if $u$ has finite distortion \cite[\S 1]{Hencl2014a}. Here, we prove that any weak extremal  has essentially constant rank, and therefore orientation-preserving extremals are either harmonic or have finite distortion:

\begin{theorem}\label{thm:UCP}
Let $u\in W^{1,2}_\loc(\Omega,\R^2)$ be a weak solution to \eqref{eq:EL} and suppose that \eqref{eq:repF}--\eqref{eq:ell} hold. If $\Omega$ is connected then there is $k\in \{0,1,2\}$ such that $\textup{rank}(\D u)=k$ a.e.\ in $\Omega$.

In particular, if $\det \D u\geq 0$ a.e.\ in $\Omega$ then either $\det \D u=0$ a.e., in which case $u$ is harmonic, or else $\det \D u>0$ a.e., in which case $u$ has finite distortion.
\end{theorem}

\begin{remark}
It is interesting to compare Theorem \ref{thm:UCP} with the known results for harmonic maps: the constancy of the rank of harmonic maps holds for two-dimensional domains \cite{Sampson1978}, but it fails in higher dimensions unless the metric is analytic \cite{Jin1991}, see also \cite[\S 4 and Problem 7.8]{Figalli2025a}.
\end{remark}

Let us now consider the case where a weak extremal $u$ of $\mb E$ has finite distortion, the case where $u$ is harmonic being trivial. Instead of working with the representation \eqref{eq:repF}, let us introduce the function $\Psi\colon [1,\infty)\to \R$, $\Psi(t) \equiv f(2t,1)$, which allows us to write
\begin{equation}
	\label{eq:defPsi}
	F(A) = \Psi(\mb K_A) \det A, \qquad \mb K_A\equiv \frac 1 2 \frac{|A|^2}{\det A}, \quad \text{whenever $\det A>0$.}
\end{equation}	
Now suppose that $u\colon \Omega\to \Omega'$ is a homeomorphism with finite distortion with inverse $f\colon \Omega'\to \Omega$. By the change of variables formula  \cite{Hencl2005}, we have
$$\mb E^*[f]\equiv \int_{\Omega'} \Psi(\mb K_f(x)) \d x = \int_\Omega \Psi(\mb K_u(x)) \det \D u(x)\d x = \mb E[u],$$
and so minimizing $\mb E$ or minimizing $\mb E^*$, among homeomorphisms, are equivalent problems.

In the quasiconformal approach to Teichm\"uller theory one considers maps which minimize 
$\|\mb K_u\|_{L^\infty}$ e.g.\ in homotopy classes. The problem of minimizing $\mb E^*$ is thus an integral analogue of this problem, as one minimizes a suitable mean of $\mb K_f$, and it has received considerable attention in the last decades, starting with the work \cite{Astala2005}, see also 	\cite{Astala2010,Iwaniec2013,Martin2022,Martin2024,Martin2020,Martin2012,Martin2017} and \cite[\S 21]{Astala2009}. 

For regularity theory it is natural to assume that $\Psi$ is smooth, increasing and convex. Under these conditions, in \cite[Conjecture 13.1]{Astala2005} and \cite[Conjecture 21.2.1]{Astala2009}, the authors say that they ``strongly believe that (...)\ the extremals are continuously differentiable'', and even $C^{1,\alpha}$-diffeomorphisms. They also note, in line with the discussion in Section \ref{sec:regtheory}, that not even partial regularity is known. Apart from the case $\Psi(t)=t$, which corresponds to the Dirichlet energy, these conjectures remain open, but see \cite{Martin2024} for the case where $\Psi$ is exponential and one works over a closed Riemann surface rather than a domain.

In this paper we show that indeed $C^1$-regularity holds even in the broader class of orientation-preserving stationary points of $\mb E$:

	\begin{theorem}\label{thm:C1}
	Let $u\in W^{1,2}(\Omega,\R^2)$ be a stationary point of $\mb E$, and assume that \eqref{eq:repF}--\eqref{eq:ell} hold.  If $\det \D u\geq 0$ a.e.\ in $\Omega$ then $u\in C^1(\Omega)$. If moreover $u\colon \Omega\to u(\Omega)$ is a homeomorphism then $u$ is a diffeomorphism in $\Omega\setminus \Sigma$.
	\end{theorem}		
	
	\begin{remark}
	Through the representation \eqref{eq:defPsi}, whenever $\Psi$ has \textit{linear growth}, the assumptions that $\Psi$ be smooth, increasing and convex are equivalent to \eqref{eq:ell}. Note that the assumption of linear growth is crucial, as it guarantees that the integrand $F$ can be extended continuously to the whole of $\R^{2\times 2}$.
	\end{remark}

To deduce Theorem \ref{thm:C1} from Theorem \ref{thm:partialreg}, we describe the behavior of stationary points around their singular set $\Sigma$ via blowups, similarly in spirit to the recent paper \cite{Lamy2026}.\ Indeed, we start from the observation that any Lipschitz blowup of $u$ at a point of $\Sigma$, that we name \emph{tangent map}, is a linear, conformal map.\ The hard step is to improve the $L^\infty$ convergence to $W^{1,\infty}$ convergence. To do so, we use \textit{topological methods} similar to those of \cite{Faraco2008,Iwaniec2013,Lamy2026} to establish the following dichotomy: 
\begin{enumerate}
\item either 0 is a tangent map, and in this case we can use a variation of the arguments from \cite{Iwaniec2013} to show that the gradient localizes close to 0, hence the map is $C^1$;
\item or else all tangent maps are invertible, and in this case one can separate the gradient from 0, hence the map becomes locally quasiconformal. In this case the arguments leading to Theorem \ref{thm:partialreg} yield smoothness. 
\end{enumerate}
Interestingly, to rule out the potential singularities at points of $\Sigma$ we only need to use the inner variational equations \eqref{eq:inner}, cf.\ Remark \ref{rem:inner}.

\subsection*{Outline and Notation}

The paper is structured as follows. In Section \ref{sec:prelims} we collect some basic facts related to conformal invariance, frame indifference and the semiconvexity notions of the vectorial Calculus of Variations. In Section \ref{sec:partialreg} we prove Theorem \ref{thm:partialreg}, assuming a key nonlinear algebraic inequality which in turn is proved in Section \ref{sec:algin}. Section \ref{sec:constantrank} contains the proof of Theorem \ref{thm:UCP}, and in fact some more general results valid for $\sigma$-harmonic mappings. Finally, in Section \ref{sec:OP} we prove Theorem \ref{thm:C1}.

\subsubsection*{Sets and topology}

Let $E \subset \R^n$ be any set. Then, $\overline{E}$ denotes its closure, $\partial E$ its topological boundary, $E^c$ its complement in $\R^n$. 
%For two sets $A,B$, we denote by $\dist(A,B)$ the distance between them.\ 
Moreover, $A \Subset B$ means that $\overline{A} \subset B$.\ The open ball of radius $r$ centered at $X$ in $\R^n$ or in $\R^{n\times m}$ is denoted by $B_r(X)$.\ If $X = 0$, we will simply write $B_r$.\ Finally, given a Lipschitz map $f: E \subset \R^k \to \R^n$, we write $\Lip(f)$ for its Lipschitz constant.

\subsubsection*{Linear Algebra}
We write $e_i$ for the vectors of the canonical basis of $\R^2$: $e_1 = (1,0)$ and $e_2 = (0,1)$. $\det(A)$, $A^T$ and $|A|$ denote the determinant, the transpose and the Euclidean norm of the matrix $A$, respectively.\ The (standard) scalar product between matrices is denoted by $\langle A,B\rangle= \tr(A^T B)$, so that $|A| = \sqrt{\langle A,A \rangle}$. Finally, the cofactor matrix is
\begin{equation}\label{eq:A}
	\cof(A) \equiv \left(\begin{array}{cc} d & -c \\ -b & a\end{array}\right), \quad \text{ if } A = \left(\begin{array}{cc} a & b \\ c & d\end{array}\right),\quad
	\text{so }\cof^T(A)A = A\cof^T(A)= \det(A)\Id\,.
\end{equation}
Let us also note, for $A,B\in \R^{2\times 2}$, the useful identity
\begin{equation}
\label{eq:quadratic}
\det(A+B)= \det A + \langle \cof A, B\rangle + \det B.
\end{equation}
In addition, we will often use some special spaces of matrices:
\begin{itemize}
	\item $\Diag^+_m$ is the set of $m\times m$ diagonal matrices with non-negative entries;
	\item $\CO = \{A \in \R^{2\times 2}: A = \cof A\}$ is the set of conformal $2\times 2$ matrices, while  $\ACO = \{A \in \R^{2\times 2}: A = -\cof A\}$ is the set of anticonformal matrices;
	\item $\OO(m) = \{A \in \R^{m\times m}: A^TA = \Id \}$ is the set of orthogonal matrices. $\SO(m) = \{A \in \OO(m): \det(A) = 1\}$ is the special orthogonal group.\ We also denote $\OO_-(2) \equiv \OO(2) \setminus \SO(2)$;
	\item $\Sym_2 = \{A \in \R^{2\times 2 }: A^T = A\}$ is the set of symmetric matrices, and $\Sym_2^+$ is the set of positive definitive symmetric matrices.
\end{itemize}

\subsubsection*{Complex derivatives}\label{sec:code}
In what follows it will be convenient to switch between real and complex notation, as some properties become more transparent in the latter.\
Let $A$ be as in \eqref{eq:A}.\ Its conformal and anti-conformal parts are, respectively:
\[
[A]_{\mathcal{H}} \equiv  \frac{1}{2}[(a + d) + i(c-b)] \quad \text{ and } \quad [A]_{\overline{\mathcal{H}}} \equiv  \frac{1}{2}[(a - d) + i(c+ b)].
\]
Direct computations show that
\begin{equation}\label{eq:alg}
	\det(A) = |[A]_{\mathcal{H}}|^2 - |[A]_{\overline{\mathcal{H}}}|^2, \quad |A|^2 = 2|[A]_{\mathcal{H}}|^2 + 2|[A]_{\overline{\mathcal{H}}}|^2.
\end{equation}
Finally, if $\Omega \subset \mathbb{C}$ is an open set and $f:\Omega \to \C$, whenever its differential exists we set
\begin{equation}\label{wirt}
	\begin{split}
		\partial_zf &= f_z \equiv [\D f]_{\mathcal{H}} = \frac{1}{2}[(\partial_1f_1 + \partial_2f_2) + i(\partial_1f_2 - \partial_2f_1)], \\
		\partial_{\bar z }f &= f_{\overline z} \equiv [\D f]_{\overline{\mathcal{H}}} = \frac{1}{2}[(\partial_1f_1 - \partial_2f_2) + i(\partial_1f_2 + \partial_2f_1)].
	\end{split}
\end{equation}

\section{Conformal invariance and frame independence}\label{sec:prelims}

In this section we give simple characterizations of conformal invariance and frame independence, introduced in \eqref{eq:ci} and \eqref{eq:fi} respectively.

\begin{lemma}\label{lemma:confinv}
The energy $\mb E$ is conformally invariant if and only if 
\begin{align}
\label{eq:condsF}
F(AR)=F(A), \qquad
F(tA)= t^2 F(A),
\end{align}
for all $A\in \R^{2\times 2}, R\in \tp {SO}(2)$ and $t>0$.
\end{lemma}

\begin{proof}
We just prove that if $\mb E$ is conformally invariant then \eqref{eq:condsF} holds, as the other direction is identical. For any $u\in W^{1,2}(\Omega,\R^2)$ and any conformal diffeomorphism $\varphi\colon \Omega'\to \Omega$ we have
\begin{align*}
 \int_{\Omega'} F(\D u\circ\varphi\, \D\varphi)\d y
 = \mb E[u\circ \varphi]
= \mb E[u] 
= \int_{\Omega} F(\D u)\d x
= \int_{\Omega'} F(\D u\circ\varphi)\det \D\varphi\d y.
\end{align*}
Taking $u$ to be a linear map,  we then deduce the two conditions in \eqref{eq:condsF} by respectively taking $\varphi$ to be a rotation and a dilation.
\end{proof}

Reasoning as above, we see that:

\begin{lemma}\label{lemma:frameind}
The energy $\mb E$ is frame indifferent if and only if 
\begin{align}
F(QA)=F(A),\quad \text{ for all $A\in \R^{2\times 2}, Q\in \tp O(2)$.}
\end{align}
\end{lemma}

We can now combine the above two lemmas to find:

\begin{lemma}\label{lemma:charact}
The following are equivalent:
\begin{enumerate}
\item the energy $\mb E$ is conformally invariant and frame indifferent;
\item the representation \eqref{eq:repF} holds.
\end{enumerate}
\end{lemma}

\begin{proof}
It is clear from the previous lemmas that any integrand as in \eqref{eq:repF} induces a conformally invariant and frame indifferent integrand, so we just prove the converse. From the singular value decomposition, any matrix $A\in \R^{2\times 2}$ can be written as
\begin{equation}
\label{eq:SVD}
A=Q_AD_A R_A, \qquad\tp{ where } Q_A\in \tp{O}(2), R_A\in \SO(2), D_A\in \Diag_2^+,
\end{equation}
thus from the lemmas we have $F(A)=F(D_A)$. If $D_A= \tp{diag}(\lambda_1, \lambda_2)$ then we can set 
$$\Phi \colon (\lambda_1, \lambda_2)\mapsto (\lambda_1^2+\lambda_2^2, \lambda_1 \lambda_2)$$
and note that $\Phi$ is a bijection from $\{(\lambda_1, \lambda_2):\lambda_1\geq \lambda_2\geq 0\}$ to $\mb O\equiv \{(n,d):n\geq 2 d\geq 0\}$, thus we can define $f\colon \mb O \to \R$ as
$$f(n,d) \equiv F(\tp{diag}(\Phi^{-1}(n,d))).$$
It follows that the representation \eqref{eq:repF} holds for all $A\in \R^{2\times 2}$ with $\det A\geq 0$.\ Let $Q = \diag(1,-1)$.\ Since $F(QA)=F(A)$, by extending $f$ in an even way, i.e.\ by setting $f(n,-d)= f(n,d)$ for $(n,d)\in \mb O$, this makes representation \eqref{eq:repF} valid for all $A$. Finally, since $F$ is positively 2-homogeneous, $f$ is positively 1-homogeneous, and the conclusion follows.
\end{proof}

Throughout this paper, we will assume that $F$ has a representation as in \eqref{eq:repF}. We denote by $f_1, f_2, f_{11},\dots$ the partial derivatives of $f$ and we write $f_i(A)\equiv f_i(|A|^2, \det A)$, and analogously for higher order derivatives, to simplify the notation. It is easy to compute the derivative
\begin{equation}
\label{eq:derF}
F'(A) = 2 f_1(A) A  + f_2(A) \cof A.
\end{equation}
For $A = 0$, the previous expression needs to be interpreted as $F'(0) = 0$.\ Thanks to the boundedness of $f_1,f_2$, $F'$ defined in this way belongs to $C^0(\R^{2\times 2},\R^{2\times 2})\cap C^\infty(\R^{2\times 2}\setminus \{0\},\R^{2\times 2})$. Either from \eqref{eq:derF}, or by differentiating directly the identity $F(QAR)=F(A)$, it is easy to see that
\begin{equation}
\label{eq:invder}
F'(QAR)=QF'(A)R \qquad \text{for all }Q,R\in \tp O(2)
\end{equation}
and all $A\in \R^{2\times 2}$.
The next lemma gives an explicit form for the stress-energy tensor:

\begin{lemma}
Let $F\colon \R^{2\times 2}\to \R$ be as in \eqref{eq:repF}. Then, with $T$ as in \eqref{eq:defT}, we have
\begin{equation}
\label{eq:simpleT}
T(A)=f_1(A)\left(|A|^2 \Id -2 A^T A\right), \text{ if }A \neq 0,
\end{equation}
and $T(0) = 0$.\ With this choice $T \in C^0(\R^{2\times 2},\R^{2\times 2})\times C^\infty(\R^{2\times 2}\setminus \{0\},\R^{2\times 2})$.
\end{lemma}

\begin{proof}
Let first $A \neq 0$.\  From \eqref{eq:derF} we compute
$$A^T F'(A) = 2 f_1(A) A^T A + f_2(A) \det(A)\,\Id.$$
Since $F$ is positively 2-homogeneous, by Euler's identity we have
$$F(A) =\frac 1 2  \langle F'(A),A\rangle=\frac 1 2 \tr(A^T F'(A)) = f_1(A)|A|^2 + f_2(A) \det(A)$$
and the conclusion follows from the definition \eqref{eq:defT}.\ Finally, if $A = 0$, expression \eqref{eq:defT} yields $T(0) = 0$.\ The continuity and the smoothness of $T$ immediately follows from \eqref{eq:defT}.
\end{proof}

The next result can be compared with Noether's theorem, as conformal invariance of the energy implies that the associated stress-energy tensor can be identified with a holomorphic differential:

\begin{corollary}\label{cor:hopf}
If $u\in W^{1,2}(\Omega,\R^2)$ is a weak solution of \eqref{eq:inner}, $\cof T(\D u)$ is locally the Hessian of a harmonic function.\ This property holds globally if $\Omega$ is simply connected.
\end{corollary}

\begin{proof}
When $\Omega$ is simply connected,  \eqref{eq:inner} is equivalent to the existence of $v\in W^{1,2}(\Omega,\R^2)$ such that $T(\D u)=\cof \D v$.\ Note that $T$ is symmetric and traceless, hence so is $\D v$: thus $v$ is the gradient of a harmonic function. 
\end{proof}

Let us also note, for later use, the following characterization of the zero level set of $T$:

\begin{lemma}\label{lemma:zeroT}
Let $A\in \R^{2\times2}$. Then $T(A)=0$ if and only if $A\in \CO \cup \ACO$. If this is the case, then
\begin{equation}
\label{eq:F'conf}
F'(A) =\left( 2 f_1(\Id) + f_2(\Id)\right) A = f(\Id)A.
\end{equation}
\end{lemma}

\begin{proof}
The characterization of the zero level set of $T$ follows immediately  from \eqref{eq:simpleT} and the assumption \eqref{eq:ell}, due to the fact that $\CO\cup \ACO=\{A\in \R^{2\times 2}:|A|^2\Id = 2 A^T A\}$. As $|A|^2 = 2\det A$ if $A \in \CO$ and $|A|^2 = -2\det A$ if $A \in \ACO$, in both cases we have $\cof A= \frac{\det A}{\lvert \det A\rvert} A$. Since $f$ is even in the last variable, $f_2$ is odd in the last variable and therefore 
$$f_2(|A|^2, \det A) \cof A = f_2(\Id) A \quad \text{for } A\in \CO\cup \ACO.$$
Hence \eqref{eq:F'conf} follows directly from \eqref{eq:derF}.\ Finally, $2 f_1(\Id) + f_2(\Id) = f(\Id)$ due to Euler's identity.
\end{proof}

We conclude this section by briefly explaining why the ellipticity of our problem is encoded through the parameter $\nu>0$ in \eqref{eq:ell}.  
Recall that an integrand $F\colon \R^{2\times 2}\to \R$ is said to be strongly rank-one convex if there is $\nu>0$ such that $t\mapsto F(A+t X)-\nu |A+ tX|^2$ is convex, for any $A,X\in \R^{2\times 2}$ with $\rank(X)=1$. This is equivalent to saying that the Euler--Lagrange system \eqref{eq:EL} satisfies the Legendre--Hadamard ellipticity condition \cite{Dacorogna2007}. We then have:

\begin{lemma}[Baker--Ericksen inequality]\label{lem:BE}
	Let $\mb U \equiv \{(n,d):n>2|d|>0\}$ and let $F\colon \R^{2\times 2}\to \R$ be an integrand of the form
	$$F(A)=f(|A|^2, \det A), \qquad f\in C^\infty(\overline {\mb U}\setminus \{0\}) \cap C^0(\overline {\mb U}).$$
If $F$ is strongly rank-one convex then there is $\nu>0$ such that $f_1 \geq \nu \text{ in } \mb U$.\\
Conversely, if $f_1\geq \nu$ in $\mb U$ and $f$ is convex, then $F$ is strongly rank-one convex.
	\end{lemma}

\begin{proof}
We compute second derivatives for points $A \neq 0$ and directions $X$ with $\det(X) = 0$:
\begin{equation}\label{eq:secder}
\frac 1 2  \frac{\d^2}{\d t^2} F(A+t X) \Big|_{t=0}= |X|^2 f_1(A)  + \frac 1 2 f''(A) [v,v],
\end{equation}
where $f''(A) \in \Sym(2)$ is the Hessian of $f$ and $v=(2 \langle A,X\rangle, \langle \cof A, X\rangle)$.\ Assuming $F$ is strongly rank-one convex, we get from \eqref{eq:secder}
\[
	\nu |X|^2 \le |X|^2 f_1(A)  + \frac 1 2 f''(A) [v,v].\
\]
By exploiting the singular value decomposition of $A$, we can find a non-zero, singular matrix $X_A$ for which $v = 0$. Thus, $f_1 \ge \nu$, as wanted.\ Conversely, if $f$ is convex and $f_1 \ge \nu$, then:
\[
|X|^2 f_1(A)  + \frac 1 2 f''(A) [v,v] \ge \nu |X|^2.
\]
From \eqref{eq:secder}, we then get that $F$ is strongly rank-one convex with parameter $\nu$ on every line parallel to a rank-one matrix which does not intersect $0$.\ By continuity of $f$ and density of these lines, $F$ is also strongly rank-one convex with parameter $\nu$ on rank one lines which intersect $0$.
\end{proof}

\section{Theorem \ref{thm:partialreg}: Partial regularity}\label{sec:partialreg}

The purpose of this section is to prove the following more precise version of Theorem \ref{thm:partialreg}:

\begin{theorem}\label{thm:precise}
	Let $u\in W^{1,2}(\Omega,\R^2)$ be a weak solution of \eqref{eq:EL}--\eqref{eq:inner} and suppose that $F$ satisfies \eqref{eq:repF}--\eqref{eq:ell}. The following hold:
	\begin{enumerate}
		\item\label{it:2} either the set $\Sigma\equiv \{T(\D u)=0\}$ is discrete or $u$ is harmonic;
		\item\label{it:lowerbound} if $0<m\leq \inf_{B_r(x_0)} |\D u|$ then $u\in C^\infty(B_r(x_0),\R^2)$, so in particular $u\in C^\infty(\Omega\setminus\Sigma)$;
	
		\item\label{it:W14} there exists $C=C(F) > 0$ such that, for every $\Omega' \Subset \Omega$ and every $h$ with $|h| < \tp{dist}(\Omega',\partial \Omega)$,
		\begin{equation}\label{eq:compact}
		 \|\D u (\cdot + h) - \D u (\cdot)\|_{L^2(\Omega')} \le C\|u\|_{W^{1,2}(\Omega)} |h|^{\frac 1 2}.
		\end{equation}
		In addition, there is $\e = \e(u)>0$ such that $u\in W^{3/2 + \e,2}(\Omega,\R^2)$.
	\end{enumerate}
\end{theorem}

\begin{remark}
Estimate \eqref{eq:compact} can be rewritten more compactly using Besov spaces, and it asserts that $\D u\in B^{3/2}_{2,\infty}(\Omega)$ with a corresponding estimate.
\end{remark}

The proof of this theorem relies on a key algebraic inequality, which is the heart of the proof:

\begin{theorem}\label{thm:keyineq}
Let $F\in C^1(\R^{2\times 2})$ be as in \eqref{eq:repF}--\eqref{eq:ell}. There are constants $C>c_F>0$ depending only on $F$ such that, for any $A,B\in \R^{2\times 2}$, we have
\begin{equation}\label{eq:comp}
c |A-B|^2 \le C\frac{|T(A)-T(B)|^2}{\max\{|A|^2,|B|^2\}} +  \langle F'(A)-F'(B),A-B\rangle .
\end{equation}
\end{theorem}

Theorem \ref{thm:keyineq} expresses a compensation phenomenon: the quantity
\begin{equation}
\label{eq:defdelta}
\Delta(A,B)\equiv \langle F'(A)-F'(B),A-B\rangle
\end{equation}
does not control $|A-B|^2$, as otherwise the integrand would be (strongly) convex.\ However, this is compensated by the first addendum on the right-hand side of \eqref{eq:comp}.\ Thanks to Corollary \ref{cor:hopf}, for a stationary point the quantity $T(\D u)$ is a priori smooth, and thus \eqref{eq:comp} can be seen as a perturbation of the usual strong monotonicity inequality valid for uniformly convex integrands.\ This \emph{compensated regularity} principle is the very same of \cite[Remark 6.2]{Tione2021}, see also \cite{Hirsch2023}, and it goes back to an idea of V. \v{S}ver{\'{a}}k \cite{Sverak1993} concerning elliptic differential inclusions.\ However, the analysis to obtain \eqref{eq:comp} in this context is much more subtle than in \cite{Tione2021}, due to the fact that the integrands under consideration have a more complicated structure than the area integrand.\ We defer the proof of Theorem \ref{thm:keyineq} to Section \ref{sec:algin}, while here we show how it implies Theorem \ref{thm:precise}. 

\subsection{Proof of Theorem \ref{thm:precise}}\label{sec:proof}

Throughout the proof we will assume, without loss of generality, that $\Omega$ is the ball $B_1$.\

Due to Corollary \ref{cor:hopf} we have $T(\D u) = \cof \D v$ where $v$ is the gradient of a harmonic function.\ If $\Sigma=\Omega$ then $\cof \D v=0$ in $B_1$, hence by Lemma \ref{lemma:zeroT} we have $F'(\D u)= f(\Id)\D u$.
By \eqref{eq:repF}-\eqref{eq:ell} $f(\Id) \neq 0$, thus \eqref{eq:EL} shows that $\Delta u = 0$. Otherwise, $\Sigma$ is discrete. This shows \ref{it:2}, and we assume throughout the rest of the proof that $\Sigma$ is discrete as otherwise there is nothing to show.

Let us start with \ref{it:lowerbound}. Note that the second claim is clear from the first one. Indeed, by \eqref{eq:ell} and \eqref{eq:simpleT}, there is a constant $C$, depending only on $F$, such that $|T(A)|\leq C |A|^2$. Therefore, since $T(\D u)$ is continuous, $|\D u|$ is essentially locally bounded from below in the complement of $\Sigma$. Let us thus prove the first claim, so  assume that  $0 < m \le |\D u|^2$ a.e.\ in $B_{2r}(x_0)$. For matrices $A,B$ with $|A|^2,|B|^2\ge m$, inequality \eqref{eq:comp} yields
\[
c|A-B|^2 \le C m^{-1}{|T(A)-T(B)|^2} +  \langle F'(A)-F'(B),A-B\rangle.
\]
Now, as $-2\det(T(A)-T(B)) = |T(A)-T(B)|^2$, the latter can be written, with $\kappa=2Cm^{-1}$, as
\begin{equation}\label{eq:modified}
c |A-B|^2 \le - \kappa \det(T(A)-T(B)) +  \langle F'(A)-F'(B),A-B\rangle.
\end{equation}

From now on, we will use the letter $C$ to denote a constant which depends only on $m,T(\D u)$ and $F$, and which may increase line by line.\ Take any $\varphi \in C^\infty_c(B_{3r/2}(x_0))$ with $\varphi=1$ in $B_r(x_0)$, any $|h| < r/2$, and denote $g_h(x) = g(x+h) - g(x)$ for any map $g$.\ We have, thanks to \eqref{eq:EL} and the fact that the determinant is a null-Lagrangian \cite[Theorem 2.3]{Muller1999a}:
\begin{equation}
\label{eq:start}
0 = \int_{B_{2r}(x_0)}-\kappa\det(\D(\varphi v_h))+ \langle F'(\D u(x+h))-F'(\D u(x)),\D(\varphi^2 u_h(x)) \rangle \d x.
\end{equation}
Identity \eqref{eq:quadratic} and the product rule yield
\begin{align*}
\int_{B_{2r}(x_0)}&-\kappa\, \varphi^2\det(\D v_h)+ \varphi^2\langle F'(\D u(x+h))-F'(\D u(x)),\D u_h(x) \rangle \d x \\
&= \int_{B_{2r}(x_0)} \kappa\, \varphi \langle \cof(\D v_h), v_h\otimes \D\varphi\rangle \d x - 2\varphi\langle F'(\D u(x+h))-F'(\D u(x)), u_h\otimes \D\varphi \rangle \d x.
\end{align*}
Since $\det(\D v_h)=\det(T( \D u(\cdot + h)-T(\D u))$, using \eqref{eq:modified} and the Lipschitzianity of $F'$ we obtain:
\[
c\int_{B_{2r}(x_0)}\varphi^2|\D u_h|^2\d x \le C\int_{B_{2r}(x_0)}|\varphi||\D v_h||v_h||\D\varphi| + 2|\varphi||\D u_h||u_h||\D\varphi|\d x.
\]
Via Young's inequality, we finally get
\begin{align*}
c\int_{B_{2r}(x_0)}\varphi^2|\D u_h|^2\d x & \le C\int_{B_{2r}(x_0)}|\varphi||\D v_h||v_h||\D\varphi| + C\int_{B_{2r}(x_0)} |u_h|^2|\D\varphi|^2\d x \\
& \leq C|h|^2 + C \int_{B_{2r}(x_0)} |u_h|^2 |\D \varphi|^2 \d x,
\end{align*}
where in the last inequality we used the harmonicity of $v$.\ From well-known properties of Sobolev functions, we then get $\D u \in W^{1,2}_{\loc}(B_r(x_0),\R^{2\times 2})$.\ In addition, a standard argument involving Gehring's Lemma shows that $\D u \in W^{1,2+\eps}_{\loc}(B_r(x_0),\R^{2\times 2})$, see \cite[Proposition 6.5]{Tione2021} for more details. 
It follows that $\D u$ is continuous, and hence
we can obtain higher regularity through a standard bootstrap argument, thanks to the smoothness and strong rank-one convexity of the integrand, provided by Lemma \ref{lem:BE}. We refer the reader to  \cite[Section 6.1]{Tione2021} and \cite{Giaquinta2012} for more details.

\medskip

We now show \ref{it:W14}.\ 
We can further assume that $0$ is the only point in $B_1=\Omega$ where $T(\D u)$ is zero.\ 
Since $v$ is harmonic there exists $0<\alpha \leq 1$ such that
\begin{equation}\label{eq:muck}
\int_{B_1}\frac{1}{|\D v|^{2\alpha}}\d x < + \infty.
\end{equation}
Choosing any $0 \le \beta \le \alpha\leq 1$, from the fact that $|T(X)| \le C|X|^2$ we can estimate \eqref{eq:comp} as
\begin{equation}\label{eq:alpha}\begin{split}
c|A-B|^2 &\le C\frac{|T(A)-T(B)|^2}{\max\{|T(A)|,|T(B)|\}} +  \langle F'(A)-F'(B),A-B\rangle \\
&\le C\frac{|T(A)-T(B)|^{1 + \beta}}{\max\{|T(A)|,|T(B)|\}^\beta} +  \langle F'(A)-F'(B),A-B\rangle.
\end{split}
\end{equation}
Considering as above a cut-off function $\varphi\in C^\infty_c(B_{3/4})$ and letting $|h| \le 1/4$, we start from
\[
\int_{B_1}\langle F'(\D u(x+h))-F'(\D u(x)),\D(\varphi^2 u_h)\rangle \d x = 0,
\]
and we proceed similarly to before to end up with
\begin{equation}\label{eq:h}
	\begin{split}
c\int_{B_1}\varphi^2|\D u_h|^2\d x &\le C\int_{B_1}\varphi^2\frac{|\D v_h|^{1 + \beta}}{\max\{|\D v(x+h)|,|\D v(x)|\}^\beta}\d x + C\int_{B_1}|\D \varphi|^2|u_h|^2\d x\\
&\le C\int_{B_1}\frac{\varphi^2|\D v_h|^{1 + \beta}}{|\D v(x)|^\beta} \d x + C\|\D u\|_{L^{2}(\Omega)}^2|h|^2,
\end{split}
\end{equation}
and the constants only depend on $F$.
If we take $\beta = 0$, the previous inequality yields
\[
c\int_{B_1}\varphi^2|\D u_h|^2\d x \le C\|\D^2v\|_{L^\infty(B_{3/4})}|h| + C\|\D u\|_{L^2(\Omega)}^2|h|^2 \le C\|\D u\|_{L^{2}(\Omega)}^2|h|.
\]
Notice that we have used the estimate $$\|\D^2v\|_{L^\infty(B_{3/4})} \le C\|\D v\|_{L^1(B_{5/6})} \le C \|\D u\|_{L^{2}(\Omega')}^2,$$ valid since $v$ is harmonic with $\cof \D v = T(\D u)$ and $|T(\D u)|\leq C |\D u|^2$ where $C=C(F)$.\ This shows \eqref{eq:compact}, with a constant $C$ independent of $u$.\ To complete the proof of \ref{it:W14} we go back to \eqref{eq:h}, this time taking $\beta=\alpha$.\ H\"older's inequality and the smoothness of $v$ yield
\begin{equation}\label{eq:h1}
\int_{B_1}\frac{\varphi^2|\D v_h|^{1 + \alpha}}{|\D v(x)|^\alpha} \d x \le \left(\int_{B_1}\frac{\varphi^2}{|\D v(x)|^{2\alpha}}\d x\right)^\frac{1}{2} \left(\int_{B_1}\varphi^2|\D v_h|^{2(1 + \alpha)}\d x\right)^\frac{1}{2}\overset{\eqref{eq:muck}}{\le} C|h|^{1+\alpha}.
\end{equation}
Hence, since $\alpha\leq 1$, \eqref{eq:h} gives us
\[
\int_{B_1}\varphi^2|\D u_h|^2\d x \le C|h|^{1+\alpha}.
\]
In turn, this shows that
\[
\int_{B_{\frac{1}{4}}}\int_{B_1}\frac{|\D u_h|^2}{|h|^{3+\alpha/2}}\d x\d h \le C\int_{B_1}|h|^{\alpha/2-2}\d h < +\infty.
\]
Thus $\D u \in W^{s,2}(B_{1/8})$ for $s = \frac{1}{2} + \frac{\alpha}{4}>\frac 1 2$,
which concludes the proof of the theorem.

\section{The algebraic inequality}\label{sec:algin}
The aim of this section is to  prove Theorem \ref{thm:keyineq}.\ As the proof is rather long and complex, we will divide it into several steps. Roughly speaking, and as mentioned in the previous section, the first addendum on the right-hand side of \eqref{eq:comp} should be thought of as a perturbation. Thus we will first prove the desired inequality \eqref{eq:comp} when $T(A)=T(B)$, and then perform a perturbative argument to prove the full result.

\subsection{Reduction to diagonal matrices}\label{sec:red}

In this first subsection, we show that the quantity $\Delta(A,B)$ introduced in \eqref{eq:defdelta} behaves well when considering the singular value decomposition of $A,B$.\ This will allow us to reduce some parts of the proof of Theorem \ref{thm:keyineq} to diagonal matrices. Precisely, we will prove the following:

\begin{proposition}\label{prop:diag}
Let $F\in C^1(\R^{2\times 2})$ be as in \eqref{eq:repF}--\eqref{eq:ell} and let $A,B\in \R^{2\times 2}$ be two matrices such that $T(A)=T(B)$, with corresponding singular value decompositions as in \eqref{eq:SVD}. Then
$$\Delta(A,B)\geq \Delta(D_A,D_B). $$
\end{proposition}

In order to prove this proposition, we begin with the following general observation:

\begin{lemma}[Special Von Neumann trace inequality]\label{lemma:VN}
For any $X,Y\in \Diag^+_m$ and $Q\in \tp O(m)$ we have
$\tr(X Q Y )\leq \tr(XY).$
\end{lemma} 

\begin{proof}
Let $X=\diag(x_1,\dots, x_m), Y=\diag(y_1, \dots,y_m)$ where $x_i,y_i\geq 0$. Since $Q$ is orthogonal, we have $|Q_{ii}|= |Q e_i\cdot e_i| \leq \lvert Qe_i\rvert \lvert e_i\rvert \leq 1$ for all $i$. Thus
$$\tr(XQY) = \sum_{i=1}^m x_i y_i Q_{ii} \leq \sum_{i=1}^m x_i y_i= \tr(XY),$$
as wished.
\end{proof}

\begin{remark}
Recall the standard Von Neumann trace inequalities, cf.\  \cite[\S 13]{Dacorogna2007}: for $X,Y\in \Diag^+_m$ with \textit{ordered entries} and $Q,R\in \tp{O}(m)$, we have
$$\tr(XQYR)\leq \tr(XY).$$
Lemma \ref{lemma:VN} is somewhat surprising in that it asserts that, if $R=\Id$, then we do not need to assume that the entries of $X,Y$ are ordered.
\end{remark}

We now take $m=2$. 
Recall the singular value decomposition in \eqref{eq:SVD}.
As a consequence of Lemma \ref{lemma:VN}, we have:

\begin{lemma}\label{lemma:commonright}
Let $F\in C^1(\R^{2\times 2})$ be as in \eqref{eq:repF}--\eqref{eq:ell} and
suppose that $A,B\in \R^{2\times 2}$ have singular value decompositions with common right-singular vectors, i.e.\ $A=Q_AD_A R$, $B=Q_B D_B R.$
Then $$ \Delta(A,B) \geq \Delta(D_A, D_B).$$
\end{lemma}

\begin{proof}
Indeed, recalling the definition \eqref{eq:defdelta},  we expand
$$\Delta(A,B)= \langle F'(A),A\rangle + \langle F'(B),B\rangle - \langle F'(A),B\rangle - \langle F'(B),A\rangle.$$
By \eqref{eq:invder}, the first term clearly satisfies
$$ \langle F'(A),A\rangle = \langle Q_A F'(D_A)R,Q_A D_AR\rangle= \langle F'(D_A),D_A\rangle,
$$
and similarly for the second term.
Thus it suffices to consider the cross terms. Again by \eqref{eq:invder} we have
$$\langle F'(A),B\rangle = \langle Q_A F'(D_A)R, Q_B D_BR\rangle = \langle F'(D_A), QD_B \rangle,
%g_d(A) \langle \cof(R_A D_A R_A), B\rangle = g_d(A) \langle \cof(R_A) \cof(D_A) \cof(R_A), B\rangle.
$$
where $Q\equiv Q_A^T Q_B$. By \eqref{eq:derF} we have
$$F'(D_A) = 2f_1(D_A)D_A + f_2(D_A) \cof(D_A)$$
and thus by \eqref{eq:ell} we also have $F'(D_A)\in \Diag_2^+$. It follows from Lemma \ref{lemma:VN} that
$$\langle F'(A),B\rangle \leq \langle F'(D_A),D_B\rangle$$
and, by swapping the roles of $A,B$, also $\langle F'(B),A\rangle \leq \langle F'(D_B),D_A\rangle$, proving the lemma.
\end{proof}

\begin{proof}[Proof of Proposition \ref{prop:diag}]	
We can suppose that $A, B\neq 0$: for instance, if $B=0$ then by homogeneity and frame-indifference of $F$ we have
$$\Delta(A,B)=\langle F'(A),A\rangle = 2 F(A) = 2 F(D_A).$$
Thus, when $A,B\neq 0$, by Lemma \ref{lemma:commonright} it suffices to show that the condition $T(A)=T(B)$ guarantees that $A,B$ have common right-singular vectors.\ To see this, let us write $C_A = A^T A, C_B = B^T B$ for their right Cauchy--Green tensors. By \eqref{eq:simpleT},  the condition $T(A)=T(B)$  is equivalent to
$$\alpha \, \Id = f_1(A) C_A-f_1(B) C_B,\; \text{ where } \alpha \equiv\frac{f_1(A)|A|^2-f_1(B)|B|^2}{2}.$$
%where $\alpha\in \R$ is given by
%$$\alpha \equiv\frac{f_1(A)|A|^2-f_1(B)|B|^2}{2}.$$
Recalling that $f_1(A),f_1(B)>0$ by \eqref{eq:ell}, by the previous identity we immediately get that $C_A$ and $C_B$ share the same eigenvectors.\ Recalling that, in the singular value decomposition $A=Q_A D_A R_A$, the rows of $R_A$ can be chosen to be eigenvectors of $C_A$, we then readily obtain the two singular value decompositions
$$A = Q_A D_A R, \qquad B = Q_B D_B R$$
for some $Q_A,Q_B\in \tp{O}(2)$, as claimed.
\end{proof}

\subsection{The quasiconformal case}\label{sec:qc}

In order to proceed, we note the following elementary result:

\begin{lemma}
	Let $F\in C^1(\R^{2\times 2})$ be as in \eqref{eq:repF}--\eqref{eq:ell}.  For $A,B\in \R^{2\times 2}\setminus \{0\}$ we have
	\begin{equation}\label{e:DELTAAB}
		\Delta(A,B)\geq 2\nu |A-B|^2+ \left( f_2(A)+ f_2(B)\right)\det(A-B).
	\end{equation}
\end{lemma}

\begin{proof}
	By \eqref{eq:derF}, we can compute
	\begin{align*}
		\Delta(A,B)& = 2\langle  f_1(A) A- f_2(B) B,A-B\rangle +  \langle  f_2(A) \cof A- f_2(B)\cof B,A-B\rangle.
	\end{align*}
	For $\alpha, \beta\in\R$, using \eqref{eq:A}--\eqref{eq:quadratic} we find the algebraic identities
	\begin{align*}
		\langle \alpha A-\beta B,A-B\rangle & = \frac{\alpha +\beta}{2} |A-B|^2 + \frac{\alpha-\beta}{2} \left(|A|^2-|B|^2\right),\\
		\langle \alpha \cof A-\beta \cof B, A-B\rangle & =( \alpha + \beta) \det (A-B) +(\alpha-\beta) \left(\det A-\det B\right),
	\end{align*}
	and so we can rewrite
	\begin{align*}
		\Delta(A,B) = & \left( f_1(A) +	 f_1(B)\right) |A-B|^2 + \left( f_2(A)+ f_2(B)\right) \det(A-B)\\
		&  + \left( f_1(A) - f_1(B)\right)(|A|^2-|B|^2)+\left(  f_2(A)- f_2(B)\right)( \det A-\det B ).
	\end{align*}
Now, the second line is non-negative by convexity of $f$.\ This fact and \eqref{eq:ell} give the claim.
\end{proof}

Thus, in order to produce lower bounds on $\Delta(A,B)$, we are naturally led to look for lower bounds on $\det(A-B)$. The next lemma, which is of independent interest, does precisely so in the case where $\det A,\det B\geq 0$ and $T(A)=T(B)$. While preparing this article we discovered that a similar result had been shown in the particular case of the Dirichlet energy in \cite[\S 9.1]{Iwaniec2013a}.

\begin{lemma}\label{l:qc}
Let $F\in C^1(\R^{2\times 2})$ be as in \eqref{eq:repF}--\eqref{eq:ell}.\ Let $A,B\in \R^{2\times 2}$ be such that  $T(A) = T(B)$ and $\det(A),\det(B) \ge 0$.\ We have
$$\frac 1 2 \frac{1-k}{1+k} |A-B|^2\leq \det(A-B), \qquad \text{where } k \equiv  \begin{cases}\sqrt{\frac{|B|^2-2 \det B}{|B|^2+2 \det B}} \in [0,1] & \text{ if $B \neq 0$},\\
 0 & \text{ if }B = 0.\end{cases}$$
\end{lemma}

\begin{proof}
For this proof it is convenient to work with complex notation.\ Set $a_+ \equiv [A]_{\mathcal{H}}$, $a_- \equiv [A]_{\overline{\mathcal{H}}}$, and analogously for $B$.\  
We can suppose that $T(A)=T(B)\neq 0$: otherwise, by Lemma \ref{lemma:zeroT} and the assumption $\det A, \det B\geq 0$, we have $A,B\in \CO$, but then by \eqref{eq:alg} we have $\frac 1 2 |A-B|^2 =\det(A-B)$ and there is nothing to show. 
Therefore, we can suppose from the start that
\begin{equation}\label{eq:intcase}
a_-,a_+,b_-,b_+ \neq 0.
\end{equation}
Then the requirement $T(A) = T(B)$ reads, in complex notation, as:
\begin{equation}
	\label{eq:modhopf}
	f_1(A)a_+\overline{a_-} = f_1(B)b_+ \overline{b_-}.
\end{equation}
Through \eqref{eq:alg} we see that our assumptions $\det A,\det B\geq 0$ are equivalent to
\begin{equation}\label{eq:qr}
|a_-| \le |a_+|,\quad |b_-| \le k |b_+|.
\end{equation}
To prove the lemma, it suffices to prove the estimate
\begin{equation}\label{eq:conc1}
|a_-|^2 +|b_-|^2 - 2 \Re(a_-\overline{b_-}) \le k(|a_+|^2 +|b_+|^2 - 2 \Re(a_+\overline{b_+})).
\end{equation}
Multiplying \eqref{eq:modhopf} by $a_-\overline{b_+}$,
and taking the real part of the expression, we get that
\[
f_1(A)|a_-|^2\Re(a_+\overline{b_+}) = f_1(B)|b_+|^2\Re(a_-\overline{b_-}).
\]
By \eqref{eq:intcase} we can rewrite this as
\begin{equation}\label{eq:RE}
\Re(a_-\overline{b_-}) = \frac{f_1(A)|a_-|^2}{f_1(B)|b_+|^2}\Re(a_+\overline{b_+}).
\end{equation}
%From \eqref{eq:modhopf}
Moreover we note that, from \eqref{eq:modhopf} and \eqref{eq:qr}, we have the estimate
\[
f_1(A)|a_-|^2 \le f_1(A)|a_+||a_-| = f_1(B)|b_+||b_-| \le kf_1(B)|b_+|^2,
\]
so that
\begin{equation}\label{e:AB}
\frac{f_1(A)|a_-|^2}{f_1(B)|b_+|^2} \le k
\end{equation}
Plugging \eqref{eq:RE} into \eqref{eq:conc1}, we deduce that \eqref{eq:conc1} is equivalent to:
\begin{equation}\label{eq:conc2}
k(|a_+|^2 +|b_+|^2) - 2\Re(a_+\overline{b_+})\left(k -\frac{f_1(A)|a_-|^2}{f_1(B)|b_+|^2}\right) \ge |a_-|^2 + |b_-|^2.
\end{equation}
Due to Young's inequality and \eqref{e:AB}, we find that
\[
- 2\Re(a_+\overline{b_+})\left(k -\frac{f_1(A)|a_-|^2}{f_1(B)|b_+|^2}\right) \ge -(|a_+|^2+|b_+|^2)\left(k -\frac{f_1(A)|a_-|^2}{f_1(B)|b_+|^2}\right).
\]
Thus, the left-hand side of \eqref{eq:conc2} is bounded below by
\begin{align*}
k(|a_+|^2 +|b_+|^2) - 2\Re(a_+\overline{b_+})\left(k -\frac{f_1(A)|a_-|^2}{f_1(B)|b_+|^2}\right) &\ge \frac{f_1(A)|a_-|^2}{f_1(B)|b_+|^2}(|a_+|^2 +|b_+|^2) \\
&= \frac{f_1(B)}{f_1(A)}|b_-|^2 +  \frac{f_1(A)}{f_1(B)}|a_-|^2,
\end{align*}
where the last line follows from \eqref{eq:modhopf}.\ To show \eqref{eq:conc2}, we only need to bound this last expression from below by $|a_-|^2 + |b_-|^2$.\ Thus, we can consider the difference between the last expression and $|a_-|^2 + |b_-|^2$:
\begin{equation}\label{eq:f1}
\frac{f_1(B)}{f_1(A)}|b_-|^2 +  \frac{f_1(A)}{f_1(B)}|a_-|^2 - |a_-|^2 - |b_-|^2 = \left(f_1(B) - f_1(A)\right)\left(\frac{|b_-|^2}{f_1(A)} - \frac{|a_-|^2}{f_1(B)}\right).
\end{equation}
To conclude \eqref{eq:conc2} and hence the proof, we only need to show that this quantity is non-negative.\ Observe that, due to \eqref{eq:repF},
\[
f_1(A) = f_1\left(\frac{|a_+|^2 + |a_-|^2}{|a_+|^2 - |a_-|^2},\frac{1}{2}\right),
\]
and, due to \eqref{eq:ell}, we also have  that $f_1$ is increasing in its first argument.\ Therefore, $f_1(B) \ge f_1(A)$ is equivalent to having 
\[
\frac{|b_+|^2 + |b_-|^2}{|b_+|^2 - |b_-|^2} \geq \frac{|a_+|^2 + |a_-|^2}{|a_+|^2 - |a_-|^2} \quad  \iff \quad  \frac{|b_-|}{|b_+|}\geq \frac{|a_-|}{|a_+|} 
\quad \iff\quad \frac{|a_-|^2}{f_1(B)} \le \frac{|b_-|^2}{f_1(A)}
\]
where the last equivalence follows from \eqref{eq:modhopf} and the fact that $f_1\geq 0$ by \eqref{eq:ell}.
Therefore the expression  in \eqref{eq:f1} is non-negative and this concludes the proof. 
\end{proof}

\subsection{The inequality over level sets of $T$}\label{sec:last}

Combining our previous results, we can show that inequality \eqref{eq:comp} holds over level sets of $T$:

\begin{proposition}\label{prop:almost}
	Let $F\in C^1(\R^{2\times 2})$ be as in \eqref{eq:repF}--\eqref{eq:ell}.\ There is a constant $c_F>0$ such that, if $A,B\in \R^{2\times 2}$ fulfill $T(A) = T(B)$, then 
	\begin{equation}\label{eq:Delta}
		\Delta(A,B) \geq c_F|A-B|^2, \quad \forall A,B \in \R^{2\times 2}.
	\end{equation}
\end{proposition}
\begin{proof}
Note that we can suppose that $T(A)=T(B)\neq 0$. Otherwise, by \eqref{eq:F'conf}, we have
$$\Delta(A,B)=c_F |A-B|^2, \qquad c_F = 2f_1(\Id) + f_2(\Id).$$
	As in the proof of Proposition \ref{prop:diag}, the relation $T(A) = T(B)$ implies that we may write $A = Q_AD_AR$, $B = Q_BD_BR$, with $D_A,D_B\in \Diag_2^+$ and $Q_A,Q_B,R \in \OO(2)$.\ Exploiting \eqref{eq:invder},\ and through Lemma \ref{lemma:commonright}, we can assume without loss of generality that 
	$$A = QD_A, \qquad B = D_B,$$ for $D_A,D_B \in \Diag_2^+$ and $Q \in \OO(2)$.\ If $Q \in \SO(2)$ then Lemma \ref{l:qc} implies that $\det(A-B) \ge 0$, and \eqref{eq:Delta} follows then from \eqref{e:DELTAAB} and the fact that $f_2(D_A),f_2(QD_B)\geq 0$ by \eqref{eq:ell}.\  Hence we can assume $Q \in \OO_-(2)$.
	
	Under the previous simplifications, we start by showing that
	\begin{equation}\label{eq:cla1}
		\Delta(A,B) \ge 0 \text{ with equality if and only if } A = B.
	\end{equation}
	From Proposition \ref{prop:diag}, we know that
	\[
	\Delta(A,B) \ge \Delta(D_A,D_B) \ge 2\nu|D_A-D_B|^2,
	\]
	where the last inequality holds in view of the positivity of the determinants of $D_A$ and $D_B$, Lemma \ref{l:qc}, \eqref{eq:ell} and \eqref{e:DELTAAB}.\ Thus we only need to understand the equality case in \eqref{eq:cla1}. Assuming $\Delta(A,B) = 0$, then from the previous inequality we find that $D_A = D_B$.\ In that case,
	\begin{equation}\label{eq:idqt}
		\Delta(A,B) \overset{\eqref{eq:invder}}{=} \langle (\Id - Q)F'(D_A),(\Id - Q)D_A \rangle = 0.
	\end{equation}
	By \eqref{eq:derF} and \eqref{eq:A}, 
	\[
	F'(D_A)D_A = 2f_1(D_A)D_A^2 + f_2(D_A)\cof(D_A)D_A = 2f_1(D_A)D_A^2 + f_2(D_A)\det(D_A)\Id.
	\]
Recall once again that $f_2$ is an odd function in the second variable with $f_2(A)\geq 0$ when $\det A\geq0$. Therefore, if $\det(D_A) > 0$, then $F'(D_A)D_A $ is a positive definite matrix, and \eqref{eq:idqt} readily shows $\Id = Q$, which is impossible since $Q \in \OO_-(2)$.\ If $\det(D_A) = 0$ then since $D_A=D_B$ also $\det(B) = 0$.\ In that case, since $f$ is even in the second variable by \eqref{eq:repF}, we see that $f_2(A) = f_2(B) = 0$, and applying \eqref{e:DELTAAB} we see that again $A = B$.\  Hence our claim \eqref{eq:cla1} holds.\ 

\medskip
	We can now turn to proving \eqref{eq:Delta}.\ Assuming the inequality is false we find sequences $A_j,B_j$ for which $T(A_j) = T(B_j)$ and 
	\begin{equation}\label{eq:contra}
		\Delta(A_j,B_j) \le \frac{1}{j} |A_j-B_j|^2, \quad \forall j \in \N.
	\end{equation}
	By the discussion at the beginning of the proof we can assume $\det(A_j)\le 0$ and $\det(B_j) \ge 0$. By homogeneity of \eqref{eq:contra}, we can suppose without loss of generality that $|A_j| +|B_j| =1$, $\forall j \in \N$.\ By considering a non-relabeled subsequence, we can also assume the existence of the limits
	\[
	A_\infty = \lim_{j\to \infty}  A_j, \quad B_\infty = \lim_{j\to \infty} B_j, \quad  Z \equiv \lim_{j\to \infty}  \frac{A_j - B_j}{|A_j - B_j|}.
	\]
	Note that, by \eqref{eq:contra}, we have $\Delta(A_\infty,B_\infty)=0$ and so by \eqref{eq:cla1} we must have $A_\infty = B_\infty\neq 0$.\ As $\det(A_j)\det(B_j) \le 0$, we deduce $\det(A_\infty) = 0$.\ Thanks to \eqref{eq:contra}, this matrix satisfies
	\begin{equation}\label{eq:Z}
		F''(A_\infty)[Z,Z] \le 0,\quad  \text{ where } F''(A_\infty)[Z,Z] = \lim_{t \to 0}\frac{1}{t}\langle F'(A_\infty + tZ)- F'(A_\infty),Z \rangle.
	\end{equation}
	To see that this yields a contradiction, we can employ \eqref{e:DELTAAB} to write:
	\begin{equation}\label{eq:Z1}
		\Delta(A_\infty + tZ, A_\infty) \ge t^22\nu |Z|^2 + t^2(f_2(A_\infty + tZ)+f_2(A_\infty))\det Z.
	\end{equation}
	Since $A_\infty \neq 0$ is singular, again by \eqref{eq:repF}--\eqref{eq:ell} we have $\lim_{t \to 0}f_2(A_\infty + tZ) = f_2(A_\infty) = 0$.\ Moreover $|Z| = 1$, and hence dividing \eqref{eq:Z1} by $t^2$ and sending $t$ to 0,  \eqref{eq:Z1} shows that $F''(A_\infty)[Z,Z]\geq 2 \nu$, which contradicts \eqref{eq:Z}.
\end{proof}

In order to treat the general case where $T(A)\neq T(B)$, we need some control on the level sets of $T$. This is provided by the following result which, together with the implicit function theorem, implies that level sets of $T$ are 2-dimensional smooth manifolds (away from the origin):

\begin{lemma}\label{lem:differential}
	For any $B \in \R^{2\times 2}, M \in \R^{2\times 2}$, define $T'(\cdot)$ as $$T'(B)[M] = \lim_{t \to 0}\frac{T(B + t M) - T(B)}{t}.$$ Then, for all $B \neq 0,$ the linear map $M \mapsto T'(B)[M]$ has two-dimensional kernel and image.\ 
\end{lemma}
\begin{proof}
	Observe that $T$ has image in the set $S_0 \subset \Sym_2$ of traceless matrices.\ Hence, we only need to show that $T'(B)$ is a surjective map onto $S_0$.\ We compute:
	\begin{equation}\label{eq:T'}
		T'(B)[M] = \left.\frac{\d}{\d t}\right|_{t = 0}f_1(B + tM)\left(B^TB - \frac{|B|^2}{2}\Id\right) + f_1(B)(M^TB + B^TM - \langle B, M \rangle\Id)
	\end{equation}
	and
	\begin{equation}\label{eq:f_1}
		\left.\frac{\d}{\d t}\right|_{t = 0}f_1(B + tM) = 2f_{11}(B)\langle B,M\rangle + f_{12}(B)\langle \cof B,M\rangle.
	\end{equation}

Since $T(QAR)=R^T T(A) R$ for all $Q,R\in \tp O(2)$,  by exploiting the singular value decomposition of $B$ we may assume that $B = \diag(a,b)$, for $a>0,b \ge 0$.\ 
If $T(B) = 0$ then $a=b$ and
	\[
	T'(B)[M] = a f_1(B)(M^T + M - \tr(M)\Id),
	\]
	which is clearly surjective.\ If $T(B)\neq 0$, we first consider $M=e_1\otimes e_2+e_2\otimes e_1$.\ In this case the derivative in 
	\eqref{eq:f_1} vanishes and so by \eqref{eq:T'} we have
	$$T'(B)[e_1\otimes e_2 + e_2\otimes e_1] = f_1(B) (a+b) \left(e_1\otimes e_2 + e_2\otimes e_1\right)$$ 
	which is a non-zero symmetric matrix by \eqref{eq:ell}; note that it has zero diagonal elements. Next we consider the direction $M = B$ itself.\
	Since $f_1$ is $0$-homogeneous, the derivative in \eqref{eq:f_1} is again zero, and hence
		\[
	T'(B)[B] = f_1(B)(2B^TB - |B|^2\Id) = 2T(B) \in S_0.
	\]
	Since $B$ is diagonal, this matrix is as well diagonal, and non-zero since $T(B) \neq 0$. It follows that $T'(B)[e_1\otimes e_2 + e_2\otimes e_2]$ and $T'(B)[B]$ span $S_0$, which concludes the proof.
\end{proof}

\subsection{Completion of the proof of Theorem \ref{thm:keyineq}}\label{sec:complete}

\begin{proof}[Proof of Theorem \ref{thm:keyineq}]
	Using the homogeneities of the terms involved, it suffices to show the statement for matrices with $|B| \le |A| = 1$.\ We assume by contradiction that we can find matrices $A_j,B_j$ with $|B_j| \le |A_j| = 1$ fulfilling
\begin{equation}\label{eq:contra2}
	j|T(A_j)-T(B_j)|^2 +  \langle F'(A_j)-F'(B_j),A_j-B_j\rangle \le \frac{1}{j} |A_j-B_j|^2.
\end{equation}
We can assume, up to non-relabeled subsequences, that
\[
A_j \to A_\infty,\quad B_j \to B_\infty\quad \text{ and }\quad \frac{A_j- B_j}{|A_j - B_j|} \to Z.
\]
Clearly, we must have from \eqref{eq:contra2} that $T(A_\infty) = T(B_\infty)$.\ Now \eqref{eq:Delta} immediately implies that we must have $A_\infty = B_\infty.$ Dividing by $|A_j-B_j|^2$ in \eqref{eq:contra2} we  deduce that
\begin{equation}\label{eq:contra21}
\lim_{j \to \infty}\frac{T(A_j) - T(B_j)}{|A_j - B_j|} = T'(A_\infty)[Z] = 0, 
%\text{ where } T'(A_\infty)[Z] = \lim_{t \to 0}\frac{T(A_\infty + t Z) - T(A_\infty)}{t},
\end{equation}
and that
\begin{equation}\label{eq:contra22}
F''(A_\infty)[Z,Z] \le 0.
\end{equation}
By Lemma \ref{lem:differential} and the implicit function theorem, we have that
\begin{equation}\label{eq:claim2}
	\text{$E\equiv \{X \in \R^{2\times 2}: T(X) = T(A_\infty)\} \setminus \{0\}$}
\end{equation}
is a two-dimensional manifold near $A_\infty \neq 0$, and
\eqref{eq:contra21} asserts that $Z \in T_{A_\infty}E$. Consider a curve $\Gamma \colon(-\eps,\eps) \to E \subset \R^{2\times 2}$ with $\Gamma(0) = A_\infty$, $\Gamma'(0) = Z$.\ We then employ \eqref{eq:Delta} to find
\[
\Delta(\Gamma(t),\Gamma(0)) = \langle F'(\Gamma(t)) - F'(\Gamma(0)), \Gamma(t)-\Gamma(0) \rangle \ge c_F |\Gamma(t) - \Gamma(0)|^2, \quad \forall t \in (-\eps,\eps).
\]
Dividing by $t^2$ and letting $t \to 0$ yields a contradiction to \eqref{eq:contra22}.\ \end{proof}

\section{Constancy of the rank for $\sigma$-harmonic maps}\label{sec:constantrank}

The purpose of this section is to prove Theorem \ref{thm:UCP}. In fact, we will prove a more general result, namely Theorem \ref{thm:UCPsigma} below.
We begin with the following simple but crucial observation, which is due to J.\ Hirsch:

	\begin{lemma}\label{lemma:sigmasys}
	For $u$ as in Theorem \ref{thm:UCP}, we have
	\begin{equation}
	\label{eq:Bsys}
	\ddiv(\D u\, \sigma) = 0,
	\end{equation}
	for some measurable $\sigma\colon \Omega\to \Sym_2$ satisfying, for some $1\leq K<\infty$ depending on $F$,
	\begin{equation}
	\label{eq:ellsigma}
	\frac{1}{K} |\xi|^2 \leq  \sigma(x)\xi\cdot \xi \leq K |\xi|^2 \quad \text{ for all } \xi\in \R^2 \text{ and a.e.\ } x\in \Omega.
	\end{equation}
	\end{lemma}	
	
	\begin{proof}
	Indeed, by \eqref{eq:derF} we may write
	$$F'(A) = A \left[2 f_1(A) \Id+ f_2(A) A^{-1} \cof(A) \right],$$
	whenever $\det A \neq 0$. We can also rewrite
	$$f_2(|A|^2, \det A) A^{-1} \cof(A)= \frac{f_2(|A|^2, \det A)}{\det A} \cof(A^T A)\geq 0,$$
	where the last inequality is due to \eqref{eq:ell} and the fact that, since $f$ is even in the last variable, $f_2$ is odd in the last variable. Thus, when $\det A\neq 0$ we may define
	$$B(A)  \equiv 2 f_1(A) \Id  + \frac{f_2(|A|^2, \det A)}{\det A} \cof(A^T A) \geq 2 \nu \Id ,$$
		using \eqref{eq:ell}.\ This definition extends continuously to singular matrices $A \neq 0$: since $f_2(n,\cdot)$ is odd, $f_2(A) = 0$ for such matrices, and hence we can set
		\[
		B(A)  \equiv 2 f_1(A) \Id  + f_{22}(|A|^2,0) \cof(A^T A) = 2 f_1(A) \Id  + f_{22}(1,0) \frac{\cof(A^T A)}{|A|^2} \ge 2\nu\Id,
		\]
		where we have used the $-1$-homogeneity of $f_{22}$. For $A=0$, we can define $B(0)=\Id$.
Finally, we let $\sigma(x) \equiv B(\D u(x))$.\ The lower bound in \eqref{eq:ellsigma} is then clear and the upper bound follows from the fact that $B$ is 0-homogeneous and locally Lipschitz  away from zero, therefore bounded.
	\end{proof}
	
	\begin{remark}\label{rem:genint}
The above calculations did not use the homogeneity of the integrand.
In fact, the conclusion is a general consequence of frame indifference and suitable convexity assumptions.
\end{remark}
	
		Lemma \ref{lemma:sigmasys} shows that the components $u^1, u^2$ of a stationary point $u=(u^1, u^2)$ are weak solution of the \textit{same} linear elliptic equation.
	Thus, following the terminology in \cite{Alessandrini2001}, we see that any stationary point is a $\sigma$-harmonic map.  We will  deduce Theorem \ref{thm:UCP} from the following:
	
	\begin{theorem}\label{thm:UCPsigma}
	Let $u\in W^{1,2}_\loc(\Omega,\R^2)$ be a weak solution of \eqref{eq:Bsys}, where $\sigma$ satisfies \eqref{eq:ellsigma}. If $\Omega$ is connected then there is $k\in \{0,1,2\}$ such that $\textup{rank}(\D u)=k$ a.e.\ in $\Omega$.
	\end{theorem}
	
	Assuming without loss of generality that $\Omega$ is simply connected, \eqref{eq:Bsys} is then equivalent to \begin{equation}
\label{eq:vj}
J\sigma \D u^j =  \D v^j, \qquad v_j\in W^{1,2}(\Omega), \qquad j=1,2,
\end{equation}	
where $J$ is the matrix corresponding to a rotation by angle $\tfrac \pi 2$.
	Let $\varphi_j \equiv u^j  + i v^j\colon \Omega\to \C$, $j=1,2$. It is well-known that the maps $\varphi_j$ are $K$-quasiregular; in fact, they solve the following $\R$-linear Beltrami equation:
	\begin{equation}
	\label{eq:R-lin}
	\p_{\bar z} \varphi_j = \mu \partial_z \varphi_j + \nu \overline{\partial_z \varphi_j}, \qquad j=1,2.
	\end{equation}
	Here we note that, crucially, the coefficients $\mu,\nu$ are independent of $j$, and they are determined explicitly in terms of the entries $(\sigma_{ab})$ of $\sigma$ through
	$$\mu = \frac{\sigma_{22}- \sigma_{11}-2 i \sigma_{12}}{\det(\Id + \sigma)}, \qquad \nu = \frac{1-\det \sigma}{\det(\Id + \sigma)},$$
	 and the ellipticity bounds \eqref{eq:ellsigma} on $\sigma$ are equivalent to the bounds
	$$|\mu| + |\nu|\leq \frac{K-1}{K+1} \quad \tp{ a.e.\ in } \Omega,$$
	see \cite[Theorem 16.1.6]{Astala2009}.
It also follows from the theory of quasiregular maps that 
$\varphi_j$ is either constant or locally invertible away from a discrete set of points in $\Omega$. In two dimensions, this can be seen as a consequence of Stoilow's factorization \cite[Corollary 5.5.4]{Astala2009}.

	Let us now suppose that we are in a small ball $B_r(z_0)\subset \Omega$ where $\varphi_1$ is invertible. Since both maps $\varphi_1, \varphi_2$ solve the same $\R$-linear Beltrami equation \eqref{eq:R-lin}, a direct calculation using the chain rule shows that the map $\psi\equiv \varphi_2 \circ \varphi_1^{-1}$ solves
	$$\p_{\bar w} \psi = \lambda(w) \Im(\p_w \psi), \qquad \lambda(w) \equiv  \frac{-2 i\nu(z)}{1+|\nu(z)|^2 -|\mu(z)|^2}, $$
	where $w=\varphi_1(z)$, cf.\ \cite[Theorem 6.1.1]{Astala2009}. 
	
\begin{remark}
Note that whenever $\det \sigma=1$ we have $\nu=\lambda=0$, and so $\psi$ is holomorphic. This is the case when $\sigma$ is induced from the area integrand by the construction in Lemma \ref{lemma:sigmasys}. We refer the reader to \cite{Hirsch2023} for a different way of constructing a holomorphic function starting from the outer variations of the area integrand.
\end{remark}
	
We will invoke the following unique continuation result for reduced Beltrami equations \cite{DePhilippis2023,Jaaskelainen2013}:
		
	\begin{theorem}\label{thm:red}
	Let $\psi \in W^{1,2}(\Omega',\C)$ be a solution of the reduced Beltrami equation
	$$\p_{\bar w} \psi = \lambda \Im(\p_w \psi), \qquad \|\lambda\|_{L^\infty(\Omega')}<1.$$
	Then  either $\Im(\p_w \psi)=0$ a.e.\ in $\Omega'$ or $\Im(\p_w \psi)\neq 0$ a.e.\ in $\Omega'$.
	\end{theorem}			

	In order to apply Theorem \ref{thm:red}, we compute:
		
	\begin{lemma}\label{lemma:im}
	We have
	$$\D \psi \circ \varphi_1= \frac{1}{\det \D \varphi_1}\begin{bmatrix}
	\sigma \D u^2 \cdot \D u^1 & \det \D u\\
	-\det (\sigma \D u) & \sigma \D u^2 \cdot \D u^1
	\end{bmatrix}
	$$
	and thus $\Im(\p_w \psi){\circ \varphi_1}=-\frac{\left(\det \sigma+1\right)}{2 \det \D \varphi_1} \det \D u$. 
	\end{lemma}		
		
	\begin{proof}
	This is a simple linear algebra calculation. If $A^1, A^2$ are the rows of $A\in \R^{2\times 2}$, we have
	$$
	A=\begin{bmatrix}
	A^1 \\ A^2
	\end{bmatrix},\qquad 
	\det A = -A^1\cdot JA^2,
	\qquad A^{-1} = \frac{1}{\det A} \begin{bmatrix}
	-J A^2 & J A^1
	\end{bmatrix}.
	$$
	Thus, using \eqref{eq:vj}, we have
	$$\D \psi \circ \varphi_1 = \D \varphi_2 (\D \varphi_1)^{-1}= 
	\begin{bmatrix}
	\D u^2 \\ J\sigma \D u^2
	\end{bmatrix}
		\begin{bmatrix}
	\D u^1 \\ J\sigma \D u^1
	\end{bmatrix}^{-1}
	= \frac{1}{\det \D \varphi_1}
		\begin{bmatrix}
	\D u^2 \\ J\sigma \D u^2
	\end{bmatrix}
	\begin{bmatrix}
	\sigma \D u^1 & J \D u^1
	\end{bmatrix}
	$$
	from which the claim follows by symmetry of $\sigma$. The final claim is immediate from \eqref{wirt}.
	\end{proof}

		\begin{proof}[Proof of Theorem \ref{thm:UCPsigma}]
		If both $\varphi_1$ and $\varphi_2$ are constant then $u$ is constant and there is nothing to prove. Thus, up to swapping the roles of $\varphi_1$ and $\varphi_2$, we may assume that $\varphi_1$ is non-constant. By the previous discussion it then suffices to prove the conclusion in any ball $B_r(z_0)\subset \Omega$ where $\varphi_1$ is invertible. In such a ball we necessarily have $\det \D \varphi_1>0$ a.e.\ by \cite[Corollary 3.7.6]{Astala2009}. By \eqref{eq:ellsigma} we have $\det \sigma\geq K^{-2} $ a.e.\ and so it  follows from Theorem \ref{thm:red} and Lemma \ref{lemma:im} that either  $\det \D u\neq 0$ a.e.\ in $B_r(z_0)$ or $\det \D u=0$ a.e.\ in $B_r(z_0)$. Thus either $\rank(\D u)=2$ a.e.\ in $B_r(z_0)$ or $\rank(\D u)\leq 1$ a.e.\ in $B_r(z_0)$. Thus we just have to show that, in the latter case, either $u$ is constant or $\rank(\D u)=1$ a.e.\ in $B_r(z_0)$. Since we assume that $\varphi_1$ is non-constant from the start, we note that, by \eqref{eq:ellsigma}, we have
$$0 < \det \D \varphi_1 = \sigma \D u^1\cdot \D u^1\leq \Lambda |\D u^1|^2$$
a.e.\ and so $\rank(\D u)=1$ a.e. in $B_r(z_0)$, as wished.
	\end{proof}

\begin{proof}[Proof of Theorem \ref{thm:UCP}]
By Lemma \ref{lemma:sigmasys}, we can apply Theorem \ref{thm:UCPsigma} to infer the essential constancy of the rank of $\D u$. For the last claim we thus see that, if $\det \D u\geq 0$ a.e.\ then either $\det \D u>0$ a.e.\ or $\det \D u=0$ a.e.\ in $\Omega$.\ In the former case, then $u$ has finite distortion by definition. In the latter case, as noticed for instance in Lemma \ref{lemma:sigmasys}, we have $f_2(\D u)=0$ a.e.\ in $\Omega$. Therefore if $\det \D u=0$ a.e.\ then, by \eqref{eq:derF}, we have 
$$F'(\D u)= 2 f_1(|\D u|^2, 0) \D u=2 f_1(2,0) \D u,$$
since $f_1$ is 0-homogeneous. As $f_1(2,0)$ is a positive constant we see that $u$ is harmonic.
\end{proof}

\section{Theorem \ref{thm:C1}: Orientation-preserving stationary points}\label{sec:OP}

The purpose of this section is to prove Theorem \ref{thm:C1}.\ As discussed in the introduction, instead of working with $F$ and $f$ as in \eqref{eq:repF}, it is more convenient to work with the representation \eqref{eq:defPsi}.
We assume that $u$ is not harmonic and $\det \D u\geq 0$ a.e.\ so that, by Theorem \ref{thm:UCP}, $u$ is a mapping of finite distortion. Therefore
the quantity $\Psi(\mb K_u)$ is well-defined a.e.\ in $\Omega$.\ Thanks to \eqref{eq:alg}, in complex notation the distortion of $u$ can be expressed in terms of the Beltrami coefficient $\mu_u$:
 \begin{equation}
 \label{eq:Kcomplex}
 \text{if } \mu_u \equiv \frac{\p_{\bar z} u}{\p_z u}, \quad \text{ then } \mb K_u = \frac{1+|\mu_u|^2}{1-|\mu_u|^2} \text{ and } |\mu_u| = \sqrt{\frac{{ \mb K_u- 1}}{{ \mb K_u + 1}}}.
 \end{equation}
 Note that $\mu_u$ is well-defined, since $\det \D u = |u_z|^2 -|u_{\bar z}|^2>0$ a.e.\ in $\Omega$.

 Through representation \eqref{eq:defPsi}, \eqref{eq:repF}--\eqref{eq:ell} are  equivalent to asserting that
\begin{equation}
	\label{eq:ellPsi}
	\Psi \text{ is smooth, convex, and } \Psi'([1,\infty]) = [\lambda,\Lambda]
\end{equation}
for some suitable  constants $0<\lambda<\Lambda<\infty$. Indeed, we can take $\lambda=2\nu$, and the upper bound follows from the identity $\Psi'(t)=2f_1(2t,1)= 2f_1(2,1/t)$ and the smoothness of $f$ away from the origin. In particular, we see that $\Psi$ has \textit{linear growth}.
\medskip

The proof of Theorem \ref{thm:C1} proceeds through the following steps:
\begin{itemize}
	\item we start by studying blowups of orientation-preserving solutions to \eqref{eq:EL}-\eqref{eq:inner} in Section \ref{sec:compa}.\ We will show that they must be conformal, linear maps;
	\item in Section \ref{sec:nbs} we rewrite the inner variation equation as a non-autonomous Beltrami system;
	\item next, in Section \ref{sec:pr} we will recall some results of \cite{Iwaniec2013} and adapt them to our context;
	\item finally, we will show Theorem \ref{thm:C1} in Section \ref{sec:proofC1}.
\end{itemize}

\subsection{Compactness and characterization of tangent maps}\label{sec:compa}

We begin by stating the following result, which was already mentioned in Remark \ref{rem:e-reg}.

\begin{proposition}\label{cor:strong}
	Let $(u_j)\subset W^{1,2}(\Omega,\R^2)$ be a sequence of weak solutions of \eqref{eq:EL}--\eqref{eq:inner} and suppose that $F$ satisfies \eqref{eq:repF}--\eqref{eq:ell}. If $u_j\w u$ weakly in $W^{1,2}_\loc(\Omega,\R^2)$ then this convergence is strong and $u$ is also a weak solution of \eqref{eq:EL}--\eqref{eq:inner}.
\end{proposition}

\begin{proof} 
	$L^2$ equiboundedness of the gradients and \eqref{eq:compact} imply that $\{\D u_j\}_j$ is precompact in $L^2$ due to the Fréchet–Kolmogorov compactness criterion.\ The statement follows.
\end{proof}

Thank to Proposition \ref{cor:strong} we are able to classify blowups of Lipschitz stationary points, see Lemma \ref{lemma:blowup} below.\
Let us first give the required definition:

\begin{definition}\label{def:rescaling}
	Let $u\in W^{1,\infty}_\loc(\Omega,\R^2)$.
	We say that $v\colon \R^2\to \R^2$ is a \textit{tangent map to $u$ at $x_0$}, and we write $v\in T_{x_0} u$, if there is a sequence $r_n\to 0$ such that
	$$u_{x_0,r_n} \equiv \frac{u(x_0 + r_n\cdot)-u(x_0)}{r_n} \to v \quad \text{ in } C^0_\loc(\R^2,\R^2).$$
\end{definition}

We make two simple remarks. First, it is clear from the definition that $u$ is differentiable at $x_0$ if and only if $T_{x_0}u$ consists of a single linear map. Second, for a Lipschitz map $u$, any $v\in T_{x_0}u$ is necessarily a \textit{globally} Lipschitz map, since 
\begin{equation}\label{eq:Lip}
	\|\D v\|_{L^\infty(\R^2)} \leq \lim_{\eps \to 0}\|\D u\|_{L^\infty(B_\e(x_0))}
\end{equation}
The next lemma gives, for Lipschitz stationary points, a description of the possible tangent maps at singular points:

\begin{lemma}\label{lemma:blowup}
	Let $u\in W^{1,\infty}(\Omega,\R^2)$ be as in Theorem \ref{thm:precise}. If $T(\D u(x_0))=0$ then any $v\in T_{x_0}u$ is a linear map which is either conformal or anti-conformal.
\end{lemma}

\begin{proof}
	Let $v\in T_{x_0} u$ be the local uniform limit of $(u_{x_0,r_n})$.
	Note that, since $T(\D u)$ is smooth by Corollary \ref{cor:hopf}, we have
	$$\lvert T(\D u_{x_0,r_n}(x)) \rvert = \lvert T(\D u(x_0+r_n x )) \rvert =  \lvert T(\D u(x_0+r_n x ))- T(\D u(x_0)) \rvert \leq C r_n |x| \to 0$$ 
	as $n\to \infty$.\ Since the sequence $(u_{r_n,x_0})_n$ is locally equi-Lipschitz and consists of stationary points, by Proposition \ref{cor:strong} it converges strongly to another stationary point $v$, which by the above inequality satisfies $T(\D v)=0$ in $\R^2$.\ Hence, by \eqref{eq:F'conf}, we find that $\Delta v = 0$.\ As $v$ is also globally Lipschitz by \eqref{eq:Lip} and $v(0) = 0$, it must be linear and the statement follows.
	\end{proof}

\subsection{Domain variations as nonlinear Beltrami systems}\label{sec:nbs}
We aim to rewrite the inner variational equations as a nonlinear Beltrami system.\ As the result is local, we may assume from the start that 
\begin{equation}\label{eq:loc}
\text{ $\Omega$ is the ball $B_1$ and that   $\{T(\D u)=0\}=\{0\}$ .}
\end{equation} By 
Corollary \ref{cor:hopf} and \eqref{wirt}, 
the system \eqref{eq:inner} is equivalent to 
\begin{equation}
\label{eq:innercomplex}
\Psi'(\mb K_u) u_z \overline{u_{\bar z}} = \phi, \qquad \phi\colon B_1\to \C \text{ is  holomorphic},
\end{equation}
since we also assume that $\det \D u >0$ a.e.\ in $B_1$. We then have the following technical result:

\begin{lemma}\label{lemma:NB}
Let $u\in W^{1,2}(B_1,\C)$ be a weak solution of \eqref{eq:innercomplex} with $\det \D u> 0$ a.e.\ in $B_1$.\ There are constants $c_\Psi, C_\Psi > 0$ depending only on $\Psi$ such that the following holds.\ Given $r,R > 0$ fulfilling
\begin{equation}\label{eq:R}
	R  \geq c_\Psi{ \sqrt{ \|\phi\|_{L^\infty(B_r)}} }
\end{equation}
there exists a function $H\colon B_r\times (\C \setminus B_R(0)) \to \C$ such that $u$ satisfies
\begin{equation}
\label{eq:NBfinal}
u_{\bar z} = H(z,u_z) \quad \text{for a.e. $z\in B_r$ such that } u_z(z)\not \in B_R(0).
\end{equation}
Moreover, $H$ enjoys the following bounds:
\begin{equation}
\begin{gathered}
\begin{split}
\label{eq:condsHfinal}
&\sup_{z\in B_r} |H(z,\xi_1)-H(z,\xi_2)| \leq 
%\frac{L}{R^2} 
L\Big|\frac{1}{ \xi_1}-\frac {1}{\xi_2}\Big| ,\\
%|\xi_1-\xi_2|,\\
&\sup_{z\in B_r} |H(z,\xi)| + \sup_{\substack{ z_1, z_2\in B_r\\z_1\neq z_2}} \frac{|H(z_1,\xi)-H(z_2,\xi)|}{|z_1-z_2|^\alpha} \leq M,
\end{split}
\end{gathered}
\end{equation}
whenever $\xi,\xi_1,\xi_2\not \in B_R$ and $\alpha\in (0,1)$. 
The constants $L$ and $M$ satisfy the bounds
%$u$ solves an equation of the form \eqref{eq:NB}, where the constants $L,M,R$ in \eqref{eq:NB}--\eqref{eq:condsH} satisfy
\begin{equation}
\label{eq:NBest}
\begin{split}
L & \leq C_\Psi \|\phi\|_{L^\infty(B_{r})},\\
M& \leq \frac{C_\Psi}{R}\left( \|\phi\|_{L^\infty(B_{r})} +\mathrm{Lip}(\phi)\, r^{1-\alpha} \right).
\end{split}
\end{equation}
%where $C_\Psi$ is a constant depending only on $\Psi$.
\end{lemma}

\begin{proof}
We may rewrite \eqref{eq:innercomplex} as
\begin{equation}
\label{eq:rewritten}
u_{\bar z} = \frac{\bar \phi}{\overline{u_z}} \frac{1}{\Psi'(\mb K_u)}
\end{equation}
and so it suffices to express $\Psi'(\mb K_u)$ in terms of $\frac{|\phi|}{|u_z|^2}$.  Taking absolute values in \eqref{eq:innercomplex} we have
\begin{equation}\label{eq:psikmu}
\Psi'(\mb K_u) \sqrt{\frac{\mb K_u - 1}{\mb K_u + 1}}\overset{\eqref{eq:Kcomplex}}{=}  	\Psi'(\mb K_u) |\mu_u| 
	%\overset{\eqref{eq:Kcomplex}}{=} \Psi'(\mb K_u) \sqrt{\frac{\mb K_u - 1}{\mb K_u + 1}} \overset{\eqref{eq:innercomplex}}{=}
	= \frac{|\phi|}{|u_z|^2},
\end{equation}
so that we may set
\[
\eta(\mb K)\equiv \Psi'(\mb K) \sqrt{\frac{\mb K-1}{\mb K+1}}
\]
and show it is invertible as a next goal.\ Recalling \eqref{eq:Kcomplex} and using \eqref{eq:ellPsi}, we have
\begin{equation}
\label{eq:dereta}
\eta'(\mb K) = \frac{\Psi'(\mb K) + (\mb K^2-1)\Psi''(\mb K)}{(\mb K-1)^{\frac 1 2}(\mb K+1)^{\frac 3 2}} > 0
\end{equation}
hence $\eta\colon [1,\infty)\to [0,\infty)$ is strictly increasing.\ Noting that $\lim_{\mb K\to \infty} \eta(\mb K)= \Lambda$ by \eqref{eq:ellPsi}, we conclude that $\eta \colon [1,\infty)\to [0,\Lambda)$ 
is invertible with smooth inverse $\eta^{-1}\colon [0,\Lambda)\to [1,\infty)$.
Through this and \eqref{eq:rewritten}--\eqref{eq:psikmu} we can then set
$$
H(z,\xi)\equiv \frac{\overline{ \phi(z)}}{\bar \xi} \Gamma \Big( \frac{|\phi(z)|}{|\xi|^2}\Big), \qquad
\Gamma(s)\equiv\frac{1}{\Psi'(\eta^{-1}(s))},$$
provided that $\xi$ fulfills
\begin{equation}
	\label{eq:rangexi}
	\frac{\|\phi\|_{L^\infty(B_{r})} }{|\xi|^2}< \Lambda,
\end{equation}
so that $\Gamma$ is well-defined.\ Given $R$ as in \eqref{eq:R} and $c_\psi \equiv 2({\Lambda})^{-1/2}$,
 any $\xi \in B_R^c$ fulfills \eqref{eq:rangexi}.\ From \eqref{eq:innercomplex},\eqref{eq:rewritten} and the definition of $\Gamma$ we see that \eqref{eq:NBfinal} holds.\

It remains to show that $H$ satisfies \eqref{eq:condsHfinal}--\eqref{eq:NBest}. We first claim the Lipschitz estimate
\begin{align*}
\label{eq:LipGamma}
|\Gamma'(s)| = \bigg| \frac{\Psi''}{\eta'(\Psi')^2 }\bigg|(\eta^{-1}(s)) \leq C_\Psi,
\end{align*}
for any $s\in [0,\Lambda)$. Indeed, writing $\mb K=\eta^{-1}(s)$, from \eqref{eq:ellPsi} and \eqref{eq:dereta} we see that on the one hand we have $|\Gamma'(s) | \leq \sqrt{3}/\lambda^2$  whenever $\mb K\geq 2$, while on the other hand
$$|\Gamma'(s)| \leq \lambda^{-3} 3 \sqrt{3} \max_{\mb K\in [1,2]} \Psi''(\mb K).$$

We now take $R$ as in \eqref{eq:R}, $\xi_1,\xi_2 \in B_R^c$ and $z\in B_r$ and we estimate
\begin{align*}
|H(z,\xi_1)-H(z,\xi_2)|& \leq |\phi(z)|\Big\lvert \Gamma \Big(\frac{|\phi(z)|}{|\xi_1|^2}\Big)\Big\rvert \Big\lvert \frac{1}{\xi_1}-\frac{1}{\xi_2}\Big\rvert  + |\phi(z)| \Big\lvert \Gamma \Big(\frac{|\phi(z)|}{|\xi_1|^2}\Big)-\Gamma \Big(\frac{|\phi(z)|}{|\xi_2|^2}\Big)\Big\rvert\frac{1}{|\xi_2|}\\
&\leq |\phi(z)| \left(\lambda^{-1}+\frac{2|\phi(z)|}{R^2} \mathrm{Lip}(\Gamma)\right)  \Big\lvert \frac{1}{\xi_1}-\frac{1}{\xi_2}\Big\rvert \\
%& \leq \frac{C_\Psi |\phi(z)| (1+|\phi(z)|)}{R^3} |\xi_1-\xi_2|
& \leq  C_\Psi |\phi(z)| \Big( 1 + \frac{|\phi(z)|}{R^2}\Big) \Big\lvert \frac{1}{\xi_1}-\frac{1}{\xi_2}\Big\rvert \\
& \overset{\eqref{eq:R}}{\le} C_\Psi|\phi(z)|\Big\lvert \frac{1}{\xi_1}-\frac{1}{\xi_2}\Big\rvert
 \end{align*}
for a new constant $C_\Psi$. This already proves the first estimate in \eqref{eq:condsHfinal}. The proof of the second estimate is similar: for $z,z_1,z_2\in B_r$ with $z_1\neq z_2$ and $|\xi| > R$ we have 
\begin{align*}
|H(z,\xi)| +  \frac{|H(z_1,\xi)-H(z_2,\xi)|}{|z_1-z_2|^\alpha}
& \leq\frac { \|\phi\|_{L^\infty(B_r)}  }{ \lambda R } + \|\nabla_z H(\cdot,\xi)\|_{L^\infty(B_r)}(2r)^{1-\alpha} \\
& \leq \frac{C_\Psi}{R}\left[  \|\phi\|_{L^\infty(B_r)}  + \|\phi'\|_{L^\infty(B_r)} r^{1-\alpha} \right] 
\end{align*}
as wished.
\end{proof}

\subsection{Preliminary results on Beltrami systems}\label{sec:pr}

In order to prove Theorem \ref{thm:C1}, we will require the following variant of  \cite[Proposition 6.1]{Iwaniec2013}:
\begin{proposition}\label{prop:IKO}
Consider the equation
\begin{equation}
\label{eq:NBflambda}
f^\lambda_{\bar z} = H(z,\lambda + f^\lambda_z) \quad \text{for a.e.\ } z\in B_r(z_0) \text{ such that } \lambda + f^\lambda_z(z)\not \in B_R(0),
\end{equation}
where $H$ satisfies, for some $\alpha\in (0,1)$, \eqref{eq:condsHfinal} for constants $L, M > 0$.
Then, there is a constant $C_\alpha$, depending only on $\a$, such that with
\begin{equation}
\label{eq:lambda0}
\lambda_0 \geq C_\alpha \max\Big\{R,M,\sqrt{L} \Big\},
\end{equation}
there is a  family $(f^\lambda)_{|\lambda| \geq \lambda_0}$ of $C^{1,\alpha}$-solutions to \eqref{eq:NBflambda} with $f^\lambda(0)=0$ and, for all $z\in B_r(z_0)$,
\begin{align}
\label{eq:IKO1}
|\D f^\lambda(z)|&\leq \lambda_0,\\
\label{eq:IKO2}
|f^{\lambda_1}(z)-f^{\lambda_2}(z)|& \leq \sqrt{L}  \,\lambda_0 \Big|\frac{\lambda_1-\lambda_2}{\lambda_1 \lambda_2}\Big|,\\
\label{eq:IKO3}
|f^\lambda_z(z)|, |f^\lambda_{\bar z}(z)|& \leq \frac 1 2 |\lambda_0|\leq |\lambda + f^\lambda_z(z)|.
\end{align}
\end{proposition}

\begin{proof}
The statement is almost identical to that in \cite[Proposition 6.1]{Iwaniec2013}, except that the authors obtain the estimate
\begin{equation}\label{eq:initlambda0}
	\lambda_0 \geq C_\alpha \max\{1,M,R,L\},
\end{equation}
cf.\ \cite[(6.20)]{Iwaniec2013a}, and also that \eqref{eq:IKO2} holds with $\lambda_0$ in place of $\sqrt{L} \lambda_0$. 
This last estimate for $\lambda_0$ is, however, not scaling-invariant. To obtain the scaling-invariant estimate, let $\tilde H(z,\xi)\equiv H(z,\sqrt{L}\xi )/\sqrt L$, and note that $\tilde H$ satisfies \eqref{eq:condsHfinal} with $\tilde L=1$, $\tilde M = M/\sqrt{L}$ and also $\tilde R = R/\sqrt{L}$. Let  $\tilde f^{\tilde \lambda}$ be the corresponding solutions provided by \cite[Proposition 6.1]{Iwaniec2013a}, defined for $|\tilde \lambda|\geq \tilde \lambda_0$, where, using \eqref{eq:initlambda0},
$$\tilde \lambda_0 \geq C_\a \max\{1,\tilde R,\tilde M\}.$$
Hence, if we let $\lambda \equiv \sqrt L \tilde \lambda$ and $f^\lambda \equiv \sqrt{L} \tilde f^{\tilde \lambda}$, which is defined for $|\lambda|\geq \sqrt{L} \tilde \lambda_0\equiv \lambda_0$, the conclusion follows.
\end{proof}

Equipped with Proposition \ref{prop:IKO}, we can argue similarly to \cite[Proposition 7.1]{Iwaniec2013a} to find: 
\begin{proposition}\label{prop:qrdifference}
Let $u$ be a solution of \eqref{eq:NBfinal}, where $H$ satisfies \eqref{eq:condsHfinal}, and let $f^\lambda$ be as in Proposition \ref{prop:IKO}. Suppose further that $u_{\bar z}\in L^\infty(B_r(z_0))$. If 
\begin{equation}\label{eq:44}
|\lambda|\geq 4\|u_{\bar z}\|_{L^\infty(B_r(z_0))}+ 4\lambda_0 + R
\end{equation}
then $g^\lambda(z) \equiv \lambda z + f^\lambda (z)-u(z)$ is 3-quasiregular.
\end{proposition}

\begin{proof}
For the reader's convenience we present the short proof here. In complex notation, the claim is that, for a.e.\ $z\in B_r(z_0)$, we have
$$|g^\lambda_{\bar z}(z)|\leq \frac 1 2 |g^\lambda_z(z)|.$$
When $|u_z(z)|\leq R$, we simply note that:
$$2 |g^\lambda_{\bar z}|\leq 2 (|f^\lambda_{\bar z}| + |u_{\bar z}|)\overset{\eqref{eq:IKO3}}{\leq} \lambda_0 + 2 \|u_{\bar z}\|_\infty \overset{\eqref{eq:44}}{\leq} |\lambda|-\frac 1 2 \lambda_0-R \overset{\eqref{eq:IKO3}}{\leq} |\lambda|-|f^\lambda_z|-|u_z|\leq |g^\lambda_z|.$$
So assume that $|u_z(z)|\geq R$. Notice first that, again by \eqref{eq:IKO3}, 
$|\lambda+f^\lambda_z|
 \geq |\lambda|-\frac{\lambda_0}{2}
 > R,$
so both \(u\) and \(\lambda + f^\lambda\) solve  \eqref{eq:NBfinal}. Consequently, by \eqref{eq:condsHfinal},
\[
|g^\lambda_{\bar z}|
 =
|H(z,\lambda+f^\lambda_z)-H(z,u_z)|\leq 
L\left|
\frac1{\lambda+f^\lambda_z}-\frac1{u_z}
\right|  \\
=
L
\frac{|\lambda+f^\lambda_z-u_z|}
     {|\lambda+f^\lambda_z|\,|u_z|}
     = L \frac{|g^\lambda_z|}     {|\lambda+f^\lambda_z|\,|u_z|}.
\]
Hence we are done if 
$|\lambda+f^\lambda_z|\,|u_z|\geq 2L$; in the opposite case, since $|\lambda|\geq 4 \lambda_0$,
\[
|u_z|
<
\frac{2L}{|\lambda+f^\lambda_z|}
\leq
\frac{2\lambda_0^2}{|\lambda|-\lambda_0/2}
\leq
\lambda_0,
\]
since we may assume that $C_\alpha\geq 1$, and then $\lambda_0\geq \sqrt{L}$, compare \eqref{eq:lambda0}. Therefore
$$2 |g^\lambda_{\bar z}|\leq 2 (|f^\lambda_{\bar z}| + |u_{\bar z}|)\overset{\eqref{eq:IKO3}}{\leq} \lambda_0 + 2 \|u_{\bar z}\|_\infty \overset{\eqref{eq:44}}{\leq} |\lambda|-\frac 3 2 \lambda_0 \overset{\eqref{eq:IKO3}}{\leq} |\lambda|-|f^\lambda_z|-|u_z|\leq |g^\lambda_z| $$
as wished.
\end{proof}

Finally, we note here the following modification of the main result of \cite{Iwaniec2013a}:

\begin{theorem}\label{thm:IKO}
Suppose that $u\in W^{1,2}(B_r(z_0),\C)$ is such that $u_{\bar z} \in L^\infty(B_r(z_0))$ and solves \eqref{eq:NBfinal},
where $H$ satisfies \eqref{eq:condsHfinal}-\eqref{eq:NBest}. 
Then $u$ is locally Lipschitz, and for a.e.\ $z\in B_{r/3}(z_0)$ we have the estimate
\begin{equation}
\label{eq:IKO}
|\D u (z)|\leq 6 \frac {  \|u -u(z_0)\|_{L^\infty(B_r(z_0))}}{r} + 6  \|u_{\bar z}\|_{L^\infty(B_r(z_0))} + 6\lambda_0,
\end{equation}
where $\lambda_0$ is as in \eqref{eq:lambda0}.
\end{theorem}

\begin{proof}
This follows either by repeating the arguments in \cite[\S 8]{Iwaniec2013a}, which lead to (8.2) therein, or by taking estimate (8.2) for granted and then arguing as in Proposition \ref{prop:IKO} to obtain the correct scaling-invariant version.
\end{proof}

\begin{remark}
For the reader unfamiliar with \cite{Iwaniec2013}, Theorem \ref{thm:IKO} may appear rather strange. In the autonomous case, i.e.\ when $H$ does not depend on $z$, \eqref{eq:NBfinal}--\eqref{eq:condsHfinal} asserts that, at points where $\D u$ is extremely large, $\D u$ is approximately an orientation-preserving conformal matrix; note that, to be able to say so, the assumption $u_{\bar z}\in L^\infty$ is crucial. Otherwise, \eqref{eq:NBfinal}--\eqref{eq:condsHfinal} entails essentially no assumption on $u$ at points with non-large gradient. Thus Theorem \ref{thm:IKO} can be read as a differential-inclusion version of the classical result \cite{Chipot1986}, see also \cite{Dolzmann2013,Scheven2009a} for further improvements.
\end{remark}

\subsection{Proof of Theorem \ref{thm:C1}}\label{sec:proofC1}

In this subsection we finally prove Theorem \ref{thm:C1}. In fact, we have a more precise result:

\begin{theorem}\label{thm:C1precise}
Let $u$ be as in Theorem \ref{thm:precise} and suppose in addition that $\det \D u\geq 0$ a.e.\ in $\Omega$. Then $u\in C^1(\Omega)\cap C^\infty(\Omega\setminus \Sigma')$, where
$$\Sigma' = \{x\in \Omega: \D u(x)=0\}\subseteq \Sigma.$$
\end{theorem}

\begin{proof}
We can assume without loss of generality that $\Sigma$ is discrete, as otherwise there is nothing to show by Theorem \ref{thm:precise}.\ Therefore we may also assume that $\Omega=B_1$ and $\Sigma=\{0\}$, so $T(\D u)(0) = 0$. We can also suppose, by Theorem \ref{thm:UCP}, that $\det \D u>0$ a.e., as otherwise $u$ is harmonic.

In $B_1$, since $\det \D u>0$ a.e., the inner variational equations \eqref{eq:inner} can be rewritten as \eqref{eq:innercomplex} and this, combined with the assumption $\det \D u = |u_z|^2-|u_{\bar z}|^2\geq 0$ and \eqref{eq:ellPsi}, yields
\begin{equation}
\label{eq:Linfty}
\lambda |u_{\bar z}|^2 \leq \Psi'(\mb K_u) \lvert u_z\rvert \lvert u_{\bar z} \rvert = |\phi|.
\end{equation}
In particular $u_{\bar z}\in L^\infty(B_1)$ and the assumption $T(\D u)(0)=0$ is equivalent to
\begin{equation}
\label{eq:zerohopf}
\phi(0)=0.
\end{equation}

Let us consider, as in Definition \ref{def:rescaling}, the set of tangent maps $T_0 u$. In particular, note that if $v\in T_0u$, then there is a sequence$(u_{0,r_n})$ of rescalings,
where $r_n\to 0$, which converges to $v$ locally uniformly and weakly in $W^{1,2}_\loc$. Since $\det \D u\geq 0$ a.e., by the weak continuity of the Jacobian determinant, see \cite[Theorem 2.3]{Muller1999a}, we also obtain that $\det \D v\geq 0$ a.e.\ in $\R^2$. 
Thus Lemma \ref{lemma:blowup} implies that $v(z)=a z$ for some $a\in \C$.\ We now distinguish two cases. 

\medskip
\textbf{Case 1:} \textit{suppose that $0\in T_0 u$}. Then we can choose $r_n\to 0$, $r_n\leq 1$, such that $u_{0,r_n}\to 0$ locally uniformly.  By passing to a subsequence, due to \eqref{eq:zerohopf} we can further suppose that $\|\phi\|_{L^\infty(B_{r_n})}^{3/8}\leq \min\{1,c_\Psi^{-1}\}$. If we set
\begin{equation}\label{eq:choiceR}
R = R_n \equiv \max\Big\{c_\Psi \|\phi\|_{L^\infty(B_{r_n})}^{\frac 1 2}, r_n^{\frac{1-\alpha}{2}}\Big\},
\end{equation}
then we see that \eqref{eq:R} is fulfilled for such $R_n$, if $r = r_n$. Hence, by Lemma \ref{lemma:NB}, there is $H_n\colon B_{r_n}\times (\C \setminus B_{R_n}(0))\to \C$ satisfying \eqref{eq:NBfinal}--\eqref{eq:NBest}. Applying Theorem \ref{thm:IKO} and \eqref{eq:lambda0}, we then see that for a.e.\ $z\in B_{r_n/3}$ we have
\begin{align}
\begin{split}
\label{eq:Dwn}
|\D u(z)|
 \leq 6 \frac{\|u - u(0)\|_{L^\infty(B_{r_n})}}{r_n} +6\lambda^{-\frac{1}{2}} \|\phi\|^{\frac 1 2}_{L^\infty(B_{r_n})} + C_\alpha\max\{R_n,M_n,  \sqrt{L_n}\};
\end{split}
\end{align}
here, the middle term on the right-hand side has been bounded using  also \eqref{eq:Linfty}.  Now, by \eqref{eq:NBest},
\begin{equation}\label{eq:na}
	\begin{split}
M_n+ \sqrt{L_n}\le C\left[ \sqrt{\|\phi\|_{L^\infty(B_{r_n})}}+ \frac{{\|\phi\|_{L^\infty(B_{r_n})}} }{R_n} +\mathrm{Lip}(\phi)\, \frac{r_n^{1-\alpha}}{R_n}\right],
\end{split}
\end{equation}
Each addendum on the right-hand side of \eqref{eq:Dwn} converges to zero: the first due to the assumption that $u_{0,r_n}\to 0$ locally uniformly; the second, due to \eqref{eq:zerohopf}; and the third by \eqref{eq:na} and our choice of $R_n$ in \eqref{eq:choiceR}.\ Hence, when $0\in T_0u$, we have shown that for all $\eps > 0$ there exists $\delta > 0$ such that $\D u \in B_\eps(0)$ a.e.\ in $B_\delta$, and in fact everywhere in $B_\delta \setminus \{0\}$ due to the regularity provided by Theorem \ref{thm:partialreg}.\  From this property one immediately deduces that $u$ is differentiable at $0$ with $\D u(0)=0$, and that $\D u$ extends continuously to $0$. 

\medskip
\textbf{Case 2:} \textit{suppose that $0\not \in T_0 u$}.\ Let $A \in \R^{2\times 2}$ be any non-zero matrix for which its associated linear map $L_A$ belongs to $T_0u$. Thus, from the discussion before Case 1, $[A]_{\mathcal H} = a$ and $[A]_{\overline{\mathcal H}}= 0$.\ Let $r_n > 0$ be, as usual, the sequence of radii for which $u_{0,r_n} \to L_A$ locally uniformly.\ As in the previous case, we can choose $R_n\to 0$ in such a way that $M_n+\sqrt{L_n}\to 0$, so that also the corresponding constant in \eqref{eq:lambda0}, which we denote by $\lambda_{0,n}$, vanishes as $n\to \infty$.\ In particular, from now on we will take $\lambda$ to lie in the annulus
\begin{equation}
\label{eq:defrho}
A_\rho\equiv \{\lambda\in \C:  \rho < |\lambda| < 4 \rho\}, \qquad \rho \equiv \frac{|a|}{32}.
\end{equation}
Since $u_{0,r_n}\to L_A$ and $\lambda_{0,n}\to 0$, we find a large $N_0$ such that 
\begin{equation}
\label{eq:closetoa}
\|u_{0,r_n}-a\|_{L^\infty(B_{1})}\leq 16 \rho\text{ and }  \lambda_{0,n}\leq \frac \rho 2, \quad \forall n\geq N_0.
\end{equation} 
In addition, by increasing $N_0$ even further, since $\|u_{\bar z}\|_{L^\infty(B_{r_n})} \to 0$ by \eqref{eq:zerohopf}, we can suppose that
$$\rho \geq 4 \|u_{\bar z}\|_{L^\infty(B_{r_n})} + 4 \lambda_{0,n} + R_n,\quad \forall n\geq N_0.$$
As in the previous case, by Lemma \ref{lemma:NB}, there is $H_{N_0}\colon B_{r_{N_0}}\times (\C \setminus B_{R_{N_0}}(0))\to \C$ satisfying \eqref{eq:NBfinal}--\eqref{eq:NBest}.
By \eqref{eq:closetoa}, for $\lambda \in A_\rho$ we can consider $f^\lambda,g^\lambda\colon B_{r_{N_0}}\to \C$ as in Propositions \ref{prop:IKO} and \ref{prop:qrdifference} respectively, obtained for $H_{N_0}$ in place of $H$, and the last displayed inequality guarantees that
\begin{equation}\label{eq:qrag}
	g^\lambda \text{ is $3$-quasiregular in $B_{r_{N_0}}$  \quad $\forall \lambda \in A_\rho$}.
\end{equation}

We claim that, if we choose $N_1\geq N_0$ large enough then there exists $R_1> 0$ such that
\begin{equation}\label{eq:deg}
	\deg(g^\lambda, B_{r_{N_1}}, y) = 1, \quad \forall \lambda \in A_\rho,\forall y \in B_{R_1}.
\end{equation}
To show this, we consider the tangent map of $g^\lambda$ at $0$ along the sequence $(r_n)$.\ Since $f\in C^{1,\alpha}$, this tangent map is
\[
g_\infty^\lambda(z) = \lambda z + \partial_zf^\lambda(0)z + \partial_{\bar z}f^\lambda(0)\bar z - a z.
\]
In view of \eqref{eq:IKO1} and \eqref{eq:defrho} we see that, when $\lambda \in A_\rho$, the first two terms on the right-hand side are bounded by $(4\rho + \lambda_{0,N_0})|z|\leq 5 \rho|z|\leq \frac{|a|}{4} |z|$ and hence $g_\infty^\lambda$ is an invertible affine map, thus
\begin{equation}\label{eq:ginf}
\deg(g_\infty^\lambda, B_{1}, y) = 1, \quad \forall \lambda\in A_\rho, \forall y \in B_{|a|/2}(0).
\end{equation}
Using again \eqref{eq:IKO1} we can bound, for all $n\geq N_0$ and again with $\lambda \in A_\rho$:
\begin{align*}
\|(g^\lambda)_{0,r_n} - g^\lambda_\infty\|_{L^\infty(B_1)} &
 \le \|u_{0,r_n} - L_A\|_{L^\infty(B_1)} + \|(f^\lambda)_{0,r_n} - L_{\D f^\lambda(0)}\|_{L^\infty(B_1)} \\
 & \le \|u_{0,r_n} - L_A\|_{L^\infty(B_1)} + 2\lambda_{0,N_0}.
\end{align*}
Thus, by the stability of the degree with respect to uniform convergence, see e.g.\ \cite[Theorem 2.3]{Fonseca1995}, and by \eqref{eq:ginf}, we can find $N_1\geq N_0$ such that for all $n\geq N_1$ we have
\begin{equation}\label{eq:degn}
	\deg((g^\lambda)_{0,r_n}, B_{1}, y) = 1, \quad \forall \lambda \in A_\rho ,\forall y \in B_{|a|/2}.
\end{equation}
Then we obtain \eqref{eq:deg} by rescaling back, where $R_1$ depends on $|a|$ and $N_1$.\

By \eqref{eq:IKO1} and Theorem \ref{thm:IKO}, $g^\lambda$ has uniformly bounded Lipschitz constant for $\lambda \in A_\rho$. Let us assume without loss of generality that $u(0) = 0$.\
Hence there is $\delta >0$ with
\begin{equation}\label{eq:2delta}
B_{2\delta} \subset (g^\lambda)^{-1}(B_{R_1})\cap B_{r_{N_1}}, \quad \forall \lambda \in A_\rho.
\end{equation}
By \eqref{eq:qrag}--\eqref{eq:deg} we then infer that $g^\lambda$ is injective on $B_{2\delta}$, for all $\lambda \in A_\rho$.\ For a proof of this fact, simply combine \eqref{eq:qrag}--\eqref{eq:deg} with \cite[(4.4)]{Lamy2026} and \cite[Proposition 4.9(4)]{Lamy2026}. Let now $ B_\delta \times (0,\delta)\equiv Q_\delta$. By the injectivity of $g^\lambda$, for any $(z,r)\in Q_\delta$ and any $\lambda \in A_\rho$, we have $g^\lambda(z + r) \neq g^\lambda(z)$, i.e.:
\begin{equation}
\label{eq:nonequality}
\frac{u(z + r)- u(z)}{r} \neq \lambda  + \frac{f^\lambda(z+r) - f^\lambda(z)}{r} \equiv G_{r,z}(\lambda), \quad \forall \lambda \in A_\rho.
\end{equation}
For any fixed $(z,r)\in Q_\delta$, we have by \eqref{eq:IKO2} that $G_{r,z}$ is continuous. Note that, by \eqref{eq:IKO1} and the fact that $\lambda_0\leq \frac \rho 2$, we have
$$\|G_{r,z}-L_{\Id} \|_{C^0(A_\rho)}< \rho,$$
and so by \cite[Theorem 2.3]{Fonseca1995}, it follows that
$$\deg(G_{r,z},A_\rho, w)= \deg (L_{\Id} , A_\rho, w) = 1\quad \text{whenever } |w|=2\rho.$$
Thus, by \cite[Theorem 2.1]{Fonseca1995}, $\p B_{2\rho} \subset G_{r,z}(A_\rho)$. Hence, from \eqref{eq:nonequality}, it follows that
$$\bigg|\frac{u(z+r)-u(z)}{r}\bigg|\neq 2 \rho\quad \forall (z,r)\in Q_\delta,$$
and since $Q_\delta$ is connected and $(z,r)\mapsto \frac{u(z+r)-u(z)}{r}$ is continuous,  we have two alternatives:
$$\text{either }\bigg| \frac{u(z + r)- u(z)}{r} \bigg| <2 \rho, \forall (z,r)\in  Q_\delta; \quad \text{ or } \bigg| \frac{u(z + r)- u(z)}{r} \bigg| >2 \rho, \forall (z,r)\in  Q_\delta.$$
However, by the choice of $\rho$ in \eqref{eq:defrho}, we have $\lim_n|\frac{u(r_n)}{r_n}| = |a|> 4\rho$ and so the former alternative cannot hold. Therefore we deduce that
$$|\D u(z)|\geq |\p_1 u(z)|\geq 2 \rho \quad \text{ for a.e.\ } z\in B_\delta,$$
and we infer directly from Theorem \ref{thm:precise}\ref{it:lowerbound} that $u\in C^\infty(B_\delta)$.
\end{proof}

\begin{remark}\label{rem:inner}
To conclude the proof of Case 2 above, due to the orientation-preserving condition one does not need to appeal to the full strength of Theorem \ref{thm:precise}. In fact, it is possible to avoid altogether the outer variations equations \eqref{eq:EL}, and rely solely on the inner variations equations \eqref{eq:inner} to conclude, so let us sketch how to do so. First, starting from Lemma \ref{l:qc}, one can argue by contradiction, as in the proof of Theorem \ref{thm:keyineq}, to show that, if $\det A,\det B\geq 0$ and moreover $B$ is $K$-quasiregular, then
$$ c |A-B|^2 \leq \det (A-B) + C \frac{|T(A)-T(B)|^2}{\max\{|A|^2,|B|^2\}}$$
where $0<c\leq C$ depend on $K,F$. Once we know $\D u$ is bounded from below in $B_\delta$, as at the end of the proof of Case 2,  we find that
$$c |\D u_h|^2 \leq \det  \D u_h + C |h|^2$$
a.e.\ in $B_{\delta/2}$, provided that also $|h|\leq \delta/2$. We can now derive a Caccioppoli inequality for $u_h$ just as in the proof of Theorem \ref{thm:precise}\ref{it:lowerbound}, and infer that $u\in W^{2,2+\e}(B_{\delta/2})$ for some $\e>0$.
\end{remark}

\begin{proof}[Proof of Theorem \ref{thm:C1}]
By Theorem \ref{thm:C1precise}, the only claim left to show is that if $u$ is a homeomorphism then it is a diffeomorphism away from $\Sigma$. 
Returning to Lemma \ref{lemma:sigmasys} and to the precise definition of $\sigma$ therein, we see that $u$ is a $\sigma$-harmonic map, and $\sigma \in \Lip_{\loc}(\Omega\setminus \Sigma)$ by Lemma \ref{lemma:sigmasys} and Theorem \ref{thm:C1precise}. Therefore, in $\Omega\setminus \Sigma$, we can apply the result of \cite{Alessandrini2018} to conclude that $u$ is a diffeomorphism.
\end{proof}

%\clearpage
%The following code condenses the bibliography
\let\oldthebibliography\thebibliography
\let\endoldthebibliography\endthebibliography
\renewenvironment{thebibliography}[1]{
  \begin{oldthebibliography}{#1}
    \setlength{\itemsep}{0.5pt}
    \setlength{\parskip}{0.5pt}
}
{
  \end{oldthebibliography}
}

	{\small
	\bibliographystyle{abbrv-andre}
	\bibliography{library}

\begin{thebibliography}{10}

\bibitem{Acerbi1987}
E.~Acerbi and N.~Fusco.
\newblock {A regularity theorem for minimizers of quasiconvex integrals}.
\newblock {\em Arch. Ration. Mech. Anal.}, 99(3):261--281, 1987.

\bibitem{Alessandrini2001}
G.~Alessandrini and V.~Nesi.
\newblock {Univalent $\sigma$-Harmonic Mappings}.
\newblock {\em Arch. Ration. Mech. Anal.}, 158(2):155--171, 2001.

\bibitem{Alessandrini2018}
G.~Alessandrini and V.~Nesi.
\newblock {Locally invertible $\sigma$–harmonic mappings}.
\newblock {\em Rend. di Mat. e delle Sue Appl.}, 39(7):1--9, 2018.

\bibitem{Astala2009}
K.~Astala, T.~Iwaniec, and G.~Martin.
\newblock {\em {Elliptic Partial Differential Equations and Quasiconformal
  Mappings in the Plane (PMS-48)}}.
\newblock Princeton University Press, 2009.

\bibitem{Astala2010}
K.~Astala, T.~Iwaniec, and G.~Martin.
\newblock {Deformations of Annuli with Smallest Mean Distortion}.
\newblock {\em Arch. Ration. Mech. Anal.}, 195(3):899--921, 2010.

\bibitem{Astala2005}
K.~Astala, T.~Iwaniec, G.~J. Martin, and J.~Onninen.
\newblock {Extremal mappings of finite distortion}.
\newblock {\em Proc. London Math. Soc.}, 91(3):655--702, 2005.

\bibitem{Ball1977}
J.~M. Ball.
\newblock {Convexity conditions and existence theorems in nonlinear
  elasticity}.
\newblock {\em Arch. Ration. Mech. Anal.}, 63(4):337--403, 1977.

\bibitem{Chipot1986}
M.~Chipot and L.~C. Evans.
\newblock {Linearisation at infinity and Lipschitz estimates for certain
  problems in the calculus of variations}.
\newblock {\em Proc. R. Soc. Edinburgh Sect. A Math.}, 102(3-4):291--303, 1986.

\bibitem{Dacorogna2007}
B.~Dacorogna.
\newblock {\em {Direct Methods in the Calculus of Variations}}, volume~78 of
  {\em Applied Mathematical Sciences}.
\newblock Springer, New York, 2007.

\bibitem{DeLellis2019}
C.~{De Lellis}, G.~{De Philippis}, B.~Kirchheim, and R.~Tione.
\newblock {Geometric measure theory and differential inclusions}.
\newblock {\em Ann. la Fac. des Sci. Toulouse Math{\'{e}}matiques},
  30(4):899--960, 2021.

\bibitem{DePhilippis2023}
G.~{De Philippis}, A.~Guerra, and R.~Tione.
\newblock {Unique continuation for differential inclusions}.
\newblock {\em Ann. l'Institut Henri Poincar{\'{e}} C, Anal. non
  lin{\'{e}}aire}, 2024.

\bibitem{Rosa2020}
A.~{De Rosa} and R.~Tione.
\newblock {Regularity for graphs with bounded anisotropic mean curvature}.
\newblock {\em Invent. Math.}, 230(2):463--507, 2022.

\bibitem{Dolzmann2013}
G.~Dolzmann, J.~Kristensen, and K.~Zhang.
\newblock {BMO and uniform estimates for multi-well problems}.
\newblock {\em Manuscripta Math.}, 140(1-2):83--114, 2013.

\bibitem{Duzaar2004}
F.~Duzaar and G.~Mingione.
\newblock {Regularity for degenerate elliptic problems via p-harmonic
  approximation}.
\newblock {\em Ann. l'Institut Henri Poincare Anal. Non Lineaire},
  21(5):735--766, 2004.

\bibitem{Evans1986}
L.~C. Evans.
\newblock {Quasiconvexity and partial regularity in the calculus of
  variations}.
\newblock {\em Arch. Ration. Mech. Anal.}, 95(3):227--252, 1986.

\bibitem{Faraco2008}
D.~Faraco and L.~Sz{\'{e}}kelyhidi.
\newblock {Tartar's conjecture and localization of the quasiconvex hull in $
  \mathbb{R}^{{2 \times 2}} $}.
\newblock {\em Acta Math.}, 200(2):279--305, 2008.

\bibitem{Figalli2017b}
A.~Figalli.
\newblock {Regularity of $\textrm{codimension-}1$ minimizing currents under
  minimal assumptions on the integrand}.
\newblock {\em J. Differ. Geom.}, 106(3):371--391, 2017.

\bibitem{Figalli2025a}
A.~Figalli, A.~Guerra, S.~Kim, and H.~Shahgholian.
\newblock {Constraint Maps: Insights and Related Themes}.
\newblock {\em La Mat.}, 5(2):26, 2026.

\bibitem{Fonseca1995}
I.~Fonseca and W.~Gangbo.
\newblock {\em {Degree theory in analysis and applications}}.
\newblock Oxford University Press, 1995.

\bibitem{Giaquinta2004}
M.~Giaquinta and S.~Hildebrandt.
\newblock {\em {Calculus of Variations I}}, volume 310 of {\em Grundlehren der
  mathematischen Wissenschaften}.
\newblock Springer, Berlin, Heidelberg, 2004.

\bibitem{Giaquinta2012}
M.~Giaquinta and L.~Martinazzi.
\newblock {\em {An Introduction to the Regularity Theory for Elliptic Systems,
  Harmonic Maps and Minimal Graphs}}.
\newblock Scuola Normale Superiore, Pisa, 2012.

\bibitem{Gruter1984}
M.~Gr{\"{u}}ter.
\newblock {Conformally invariant variational integrals and the removability of
  isolated singularities}.
\newblock {\em Manuscripta Math.}, 47(1-3):85--104, 1984.

\bibitem{GuerraKristensen2021}
A.~Guerra and J.~Kristensen.
\newblock {Automatic Quasiconvexity of Homogeneous Isotropic Rank-One Convex
  Integrands}.
\newblock {\em Arch. Ration. Mech. Anal.}, 245(1):479--500, 2022.

\bibitem{GuerraTione2024}
A.~Guerra and R.~Tione.
\newblock {Regularity and compactness for critical points of degenerate
  polyconvex energies}.
\newblock {\em Arch. Ration. Mech. Anal.}, 248(6):107, 2024.

\bibitem{Hencl2014a}
S.~Hencl and P.~Koskela.
\newblock {\em {Lectures on Mappings of Finite Distortion}}, volume 2096 of
  {\em Lecture Notes in Mathematics}.
\newblock Springer International Publishing, Cham, 2014.

\bibitem{Hencl2005}
S.~Hencl, P.~Koskela, and J.~Onninen.
\newblock {A note on extremal mappings of finite distortion}.
\newblock {\em Math. Res. Lett.}, 12(2-3):231--237, 2005.

\bibitem{Hildebrandt1982}
S.~Hildebrandt.
\newblock {Nonlinear Elliptic Systems and Harmonic Mappings}.
\newblock In {\em Proc. 1980 Beijing Symp. Diff. Geom. Diff. Equ., Vol. 1},
  pages 481--615. Science Press, Beijing, 1982.

\bibitem{Hirsch2023}
J.~Hirsch, C.~Mooney, and R.~Tione.
\newblock {On the Lawson-Osserman conjecture}.
\newblock {\em arXiv:2308.04997}, pages 1--18, 2023.

\bibitem{Hirsch2021a}
J.~Hirsch and R.~Tione.
\newblock {On the constancy theorem for anisotropic energies through
  differential inclusions}.
\newblock {\em Calc. Var. Partial Differ. Equ.}, 60(3):1--52, 2021.

\bibitem{Iwaniec2013}
T.~Iwaniec, L.~V. Kovalev, and J.~Onninen.
\newblock {Lipschitz regularity for inner-variational equations}.
\newblock {\em Duke Math. J.}, 162(4):643--672, 2013.

\bibitem{Iwaniec2001}
T.~Iwaniec and G.~Martin.
\newblock {\em {Geometric Function Theory and Non-linear Analysis}}.
\newblock Clarendon Press, 2001.

\bibitem{Iwaniec2013a}
T.~Iwaniec and J.~Onninen.
\newblock {Mappings of Least Dirichlet Energy and their Hopf Differentials}.
\newblock {\em Arch. Ration. Mech. Anal.}, 209(2):401--453, 2013.

\bibitem{Jaaskelainen2013}
J.~J{\"{a}}{\"{a}}skel{\"{a}}inen.
\newblock {On reduced Beltrami equations and linear families of quasiregular
  mappings}.
\newblock {\em J. fur die Reine und Angew. Math.}, 682(682):49--64, 2013.

\bibitem{Jin1991}
Z.~Jin and J.~L. Kazdan.
\newblock {On the rank of harmonic maps}.
\newblock {\em Math. Zeitschrift}, 207(1):535--537, 1991.

\bibitem{Kristensen2007}
J.~Kristensen and G.~Mingione.
\newblock {The Singular Set of Lipschitzian Minima of Multiple Integrals}.
\newblock {\em Arch. Ration. Mech. Anal.}, 184(2):341--369, 2007.

\bibitem{Lamy2026}
X.~Lamy and R.~Tione.
\newblock {Hyperbolic regularization effects for degenerate elliptic
  equations}.
\newblock {\itshape Preprint}, 0, 2026,
  \href{http://arxiv.org/abs/2601.04753}{{\ttfamily arXiv:2601.04753}}.

\bibitem{Marsden1994}
J.~E. Marsden and T.~J.~R. Hughes.
\newblock {\em {Mathematical foundations of elasticity}}.
\newblock Courier Corporation, 1994.

\bibitem{Martin2022}
G.~Martin and C.~Yao.
\newblock {On the uniqueness of extremal mappings of finite distortion}.
\newblock {\em arXiv 2207.05935}, 2022.

\bibitem{Martin2024}
G.~Martin and C.~Yao.
\newblock {The exponential Teichm\"uller theory: Ahlfors--Hopf differentials
  and diffeomorphisms}.
\newblock {\itshape Preprint}, 2024,
  \href{http://arxiv.org/abs/2410.22667}{{\ttfamily arXiv:2410.22667}}.

\bibitem{Martin2020}
G.~Martin and C.~Yao.
\newblock {The $L^p$ Teichm{\"{u}}ller Theory: Existence and Regularity of
  Critical Points}.
\newblock {\em Arch. Ration. Mech. Anal.}, 248(2):13, 2024.

\bibitem{Martin2012}
G.~J. Martin and M.~McKubre-Jordens.
\newblock {Deformations with smallest weighted $L^p$ average distortion and
  nitsche-type phenomena}.
\newblock {\em J. London Math. Soc.}, 85(2):282--300, 2012.

\bibitem{Martin2017}
R.~J. Martin, I.~D. Ghiba, and P.~Neff.
\newblock {Rank-one convexity implies polyconvexity for isotropic, objective
  and isochoric elastic energies in the two-dimensional case}.
\newblock {\em Proc. R. Soc. Edinburgh Sect. A Math.}, 147(3):571--597, 2017.

\bibitem{Mazowiecka2018b}
K.~Mazowiecka and A.~Schikorra.
\newblock {Fractional div-curl quantities and applications to nonlocal
  geometric equations}.
\newblock {\em J. Funct. Anal.}, 275(1):1--44, 2018.

\bibitem{Muller1999a}
S.~M{\"{u}}ller.
\newblock {Variational models for microstructure and phase transitions}.
\newblock In {\em Calc. Var. Geom. Evol. Probl.}, pages 85--210. Springer,
  Berlin, Heidelberg, 1999.

\bibitem{Muller2003}
S.~M{\"{u}}ller and V.~{\v{S}}ver{\'{a}}k.
\newblock {Convex integration for Lipschitz mappings and counterexamples to
  regularity}.
\newblock {\em Ann. Math.}, 157(3):715--742, 2003.

\bibitem{Parker1996}
T.~H. Parker.
\newblock {Bubble tree convergence for harmonic maps}.
\newblock {\em J. Differ. Geom.}, 44(3):595--633, 1996.

\bibitem{Riviere2007}
T.~Rivi{\`{e}}re.
\newblock {Conservation laws for conformally invariant variational problems}.
\newblock {\em Invent. Math.}, 168(1):1--22, 2007.

\bibitem{Riviere2008}
T.~Rivi{\`{e}}re and M.~Struwe.
\newblock {Partial regularity for harmonic maps and related problems}.
\newblock {\em Commun. Pure Appl. Math.}, 61(4):451--463, 2008.

\bibitem{Sampson1978}
J.~H. Sampson.
\newblock {Some properties and applications of harmonic mappings}.
\newblock {\em Ann. Sci. l'{\'{E}}cole Norm. sup{\'{e}}rieure}, 11(2):211--228,
  1978.

\bibitem{Scheven2009a}
C.~Scheven and T.~Schmidt.
\newblock {Asymptotically regular problems II: Partial Lipschitz continuity and
  a singular set of positive measure}.
\newblock {\em Ann. della Sc. Norm. - Cl. di Sci.}, 8(3):469--507, 2009.

\bibitem{Schikorra2010}
A.~Schikorra.
\newblock {A remark on gauge transformations and the moving frame method}.
\newblock {\em Ann. l'Institut Henri Poincare Non Linear Anal.},
  27(2):503--515, 2010.

\bibitem{Schikorra2015}
A.~Schikorra.
\newblock {Integro-Differential Harmonic Maps into Spheres}.
\newblock {\em Commun. Partial Differ. Equations}, 40(3):506--539, 2015.

\bibitem{Sharp2013}
B.~Sharp and P.~Topping.
\newblock {Decay estimates for Rivi{\`{e}}re's equation, with applications to
  regularity and compactness}.
\newblock {\em Trans. Am. Math. Soc.}, 365(5):2317--2339, 2012.

\bibitem{Spadaro2009}
E.~N. Spadaro.
\newblock {Non-Uniqueness of Minimizers for Strictly Polyconvex Functionals}.
\newblock {\em Arch. Ration. Mech. Anal.}, 193(3):659--678, 2009.

\bibitem{Sverak1993}
V.~{\v{S}}ver{\'{a}}k.
\newblock {On Tartar's conjecture}.
\newblock {\em Ann. l'Institut Henri Poincar\'e Non Linear Anal.},
  10(4):405--412, 1993.

\bibitem{Sverak2025}
V.~{\v{S}}ver{\'{a}}k.
\newblock {New and Old Observations About Morrey's Quasi-Convexity}.
\newblock {\em \url{https://www.youtube.com/watch?v=q7qfJaI_kTc}}, 2025.

\bibitem{Szekelyhidi2004}
L.~Sz{\'{e}}kelyhidi.
\newblock {The Regularity of Critical Points of Polyconvex Functionals}.
\newblock {\em Arch. Ration. Mech. Anal.}, 172(1):133--152, 2004.

\bibitem{Tione2021}
R.~Tione.
\newblock {Minimal graphs and differential inclusions}.
\newblock {\em Commun. Partial Differ. Equations}, 46(6):1162--1194, 2021.

\bibitem{Tione2022}
R.~Tione.
\newblock {Critical Points of Degenerate Polyconvex Energies}.
\newblock {\em SIAM J. Math. Anal.}, 55(4):3205--3225, 2023.

\end{thebibliography}
	}

\end{document}